\documentclass[12pt]{amsart}
\usepackage{tikz-cd}
 \usetikzlibrary{decorations.pathmorphing}
\usepackage{amsmath, amssymb, amsthm, amsfonts, bbm, dsfont}
\usepackage{graphicx}
\usepackage{float}
\usepackage{wrapfig}
\restylefloat{figure}
\usepackage{mathrsfs}
\usepackage{todonotes}
\usepackage{caption}
\usepackage{hyperref}

\usepackage{a4wide}

\numberwithin{equation}{section}
\theoremstyle{plain}
\newtheorem{theorem}[subsection]{Theorem}
 \newtheorem{lemma}[subsection]{Lemma}
 \newtheorem{proposition}[subsection]{Proposition}
 \newtheorem{corollary}[subsection]{Corollary}

\newtheorem{question}[subsection]{Question}

 \theoremstyle{definition}

\newtheorem{remark}[subsection]{Remark}

\usepackage{stackengine}
\usepackage{mathtools}
\usepackage{amssymb}
\usepackage{dsfont}

\usepackage{mathabx}

\newcommand{\BA}{\mathbb{A}}
\newcommand{\bbA}{\mathbb{A}}
\newcommand{\CC}{\mathbb{C}}

\newcommand{\BL}{\mathbb{L}}
\newcommand{\LL}{\mathbb{L}}

\newcommand{\C}{\mathbb{C}}

\newcommand{\BZ}{\mathbb{Z}}
\newcommand{\BN}{\mathbb{N}}
\newcommand{\BQ}{\mathbb{Q}}

\newcommand{\calc}{\mathcal{C}}

\newcommand{\CO}{\mathcal{O}}

\newcommand{\CT}{\mathcal{T}}
\newcommand{\CX}{\mathcal{X}}

\newcommand{\CZ}{\mathcal{Z}}
\newcommand{\CY}{\mathcal{Y}}
\newcommand{\CI}{\mathcal{I}}

\newcommand{\CM}{\mathcal{M}}
\newcommand{\CR}{\mathcal{R}}
\newcommand{\frm}{\mathfrak{m}}

\newcommand{\spn}{\mathrm{span}}
\newcommand{\Conj}{\mathrm{Conj}}

\newcommand{\supp}{\textup{supp}}

\newcommand{\codim}{\textup{codim}}

\newcommand\Gal{\textup{Gal}}

\newcommand\Spec{\textup{Spec}}

\newcommand\Sym{\textup{Sym}}

\newcommand\Exp{\textup{Exp}}

\newcommand\Hom{\textup{Hom}}

\newcommand{\quash}[1]{}

\usepackage{tikz}
\usetikzlibrary{shapes,fit,intersections}
\usetikzlibrary{decorations.markings}
\usetikzlibrary{decorations.pathreplacing}
\usetikzlibrary{patterns}
\usetikzlibrary{matrix,arrows}
\usepgflibrary{decorations.pathmorphing}

\usepackage{tikz-cd}

\newcommand{\Hilb}{\textup{Hilb}}

\def\Hilb{ \mathrm{Hilb}}

\newcommand{\ord}{\mathrm{ord}}

\newcommand{\Def}{\mathrm{Def}_{\mathrm{acf},k}}
\newcommand{\GDef}{\mathrm{GDef}_{\mathrm{acf},k}}

\title{Discriminants and motivic integration}
\author{Oscar Kivinen}

\author{Alexei Oblomkov}

\author{Dimitri Wyss}
\begin{document}

\begin{abstract} We study invariants of a plane cuve singularity $(f,0)$ coming from motivic integration on symmetric powers of a formal deformation of $f$. We show that a natural discriminant integral recovers the motivic classes of the principal Hilbert schemes of points on $f$, while the orbifold integral gives the plethystic exponential of the motivic Igusa zeta function of $f$. The latter result also holds in higher dimemsions. 

Combined with results of Gorsky and Némethi we obtain an interpretation  of the discriminant integrals in terms of knot Floer homology, which is reminiscent of the relation between the cohomology of contact loci and fixed point Floer homology proven by de la Bodega and Poza.

\end{abstract}

\maketitle

\section{Introduction}
\label{sec:introduction}

Consider a nonconstant polynomial map $f:\BA^m \to \BA^1$ over $k=\CC$ defining an isolated hypersurface singularity at $0\in \BA^m$. A fundamental invariant of $(f,0)$ is its Milnor fiber $F_0$, a compact smooth manifold with boundary depending only on the singularity at $0$ \cite{Mi16}. Its singular cohomology $H^*(F_0,\BQ)$ admits a monodromy endomorphism $M$, whose zeta function $\zeta_{f,0}(T) = \prod_{i\geq i} \det (1-TM \ |\ H^i(F_0,\BQ))^{(-1)^{i+1}}$
can be computed purely algebraically by a remarkable result of Denef and Loeser \cite{DenefLoeser02}:

For any $n\geq 1$ let $X^0_{f,n}$ denote the $n$-th restricted contact locus of $f$. It is defined as
\begin{equation*}
X^0_{f,n} = \{ x \in \BA^m(k[[t]]/t^{n+1})_0 \ |\ f(x) = t^n \in k[[t]]/t^{n+1} \},
\end{equation*}
where $\BA^m(k[[t]]/t^{n+1})_0 $ denotes the locus of $n$-jets reducing to $0$ modulo $t$. Then the Euler characteristics $\chi(X^0_{f,n})$ determine $\zeta_{f,0}(T)$ via the formula
\begin{equation}\label{czeta} \zeta_{f,0}(T) =  \exp\left(\sum_{n\geq 1} \frac{\chi(X^0_{f,n})}{n} T^n\right).  \end{equation}

While the original proof of \eqref{czeta} proceeds by computing both sides independently on an embedded resolution of $f$, there exists now a more conceptual interpretation \cite{NS07, HrushovskiLoeser15} as a Lefschetz fixed point formula on the analytic Milnor fiber. In a similar spirit, the goal of this paper is to study refinements of \eqref{czeta} coming from motivic integration. 

Concretely, consider the $k[[t]]$-scheme
\[ \CX_f = \Spec(k[[t]][x_1,\dots,x_m] / f-t), \]
and associated with it for any $n\geq 0$ the relative symmetric power $\Sym_n\CX_f = \CX_f \times_{k[[t]]} \dots \times_{k[[t]]} \CX_f/S_n$. In a nutshell our paper studies motivic classes in the localized Grothendieck ring $\calc_k=K_0(Var_k)[\BL^{-1}, (1-\BL^i)^{-1}:i \geq 1]$ coming from Cluckers-Loeser motivic integrals \cite{CL08} on
\[\Sym_n\CX_f^0(k[[t]]) =  \{ x \in \Sym_n(\CX_f)(k[[t]]) \ |\ x_k = n[0]   \}.\]

Our first main result is in the case $m=2$, i.e. when $(f,0)$ defines an isolated plane curve singularity, and is partially motivated by the second author's work on the relation between Hilbert schemes and knot invariants \cite{OblomkovShende12, OblomkovRasmussenShende12}. Namely, since $\CX_{f|k((t))}$ is a smooth curve, the generic fiber of $\Sym_n\CX_f$ admits another model over $\Spec(k[[t]])$ given by the relative Hilbert scheme of points $ \Hilb_n(\CX_f)$. 

The special fiber $\Hilb_n(\CX_f)_k$ is the classical Hilbert scheme $\Hilb_n(f)$ of length $n$ subschemes on the singular curve $f=0$ and we write $\Hilb_n^1(f)_0\subset \Hilb_n(f)$ for the locally closed subscheme parametrizing principal subschemes supported at $0$. 

\begin{theorem}[\ref{gelhil}]\label{ingel} For any $n\geq 1$ there exists a generically non-vanishing definable volume form $|\omega_n^{gel}|$ on $\Sym_n\CX_f^0$ such that 
\[ \int_{\Sym_n \CX_f^0} |\omega^{gel}_n| = \BL^{-n} [\Hilb_n^1(f)_0] \in \calc_k. \]
\end{theorem}
By Proposition \ref{prop:motivicequalsgel}, the generating series 
\[P_{gel}(T) = \sum_{n \geq 0} \int_{\Sym_n \CX_f^0} |\omega^{gel}_n| T^n,\]
can be identified with the one-variable motivic Poincaré series of Campillo, Delgado and Gusein-Zade \cite{CDGZ03,CDGZ04}. 
%\todo{Could this part go into Section 5? For example into the proof fo Corollary 5.4?}
%Recall that this series is defined as follows when $m=2$. If the germ $(C,0)$ defined by $f$ contains $r$ branches, we can define a multi--index filtration on $k[[x,y]]$ by setting $$J_{\underline{v}}=
%\{g\in k[[x,y]]|v_i(g)\geq v_i\}$$ where $\underline{v}=(v_1,\ldots,v_r)\in \BZ^r$  and $v_i(g)$ is the order of the pullback of $f$ to the normalization of the $i$:th branch of $C$. Let $h(\underline{v})=\codim_{k[[x,y]]}J_{\underline{v}}$ and write
%$$L_f(t_1,\ldots,t_r,\BL)=\sum_{\underline{v}\in \BZ^r}\frac{\BL^{-h(\underline{v})}-\BL^{-h(\underline{v}+\underline{1})}}{1-\BL^{-1}}t^{\underline{v}}$$ By \cite{CDGZ04} one can define the motivic Poincar\'e series of $f$ as $$P_f(t_1,\ldots,t_r;\BL)=\frac{L_f(t_1,\ldots,t_r,\BL)\prod_{i=1}^r(t_i-1)}{t_1\cdots t_r-1}$$ In particular, when $r=1$ and the curve is unibranch, we have $L_f(t,\BL)=P_f(t,\BL)$. Moreover, by arguments analogous to the proof of \cite[Proposition 6]{OblomkovShende12}, for $g\in k[[x,y]]/f$ one has $\dim k[[x,y]]/(f,g)=\sum_{i=1}^r v_i(g)$ so that $$L_f(t=T,\BL)=\sum_{n=1}^\infty \BL^{-n} [\Hilb_n^1(f)_0] T^n=P_{gel}(T)$$

 Since $|\omega^{gel}_n|$ is generically non-vanishing, one can compute the Euler characteristic specialization of $P_{gel}(T) $ by the trace formula for motivic Serre invariants, recovering an old result of Milnor \cite[Lemma 10.1]{Mi16}:

\begin{corollary}[\ref{miln}] For any reduced plane curve singularity $(f,0)$ with link $L$ we have 
\[ \zeta_{f,0}(T) = \frac{A_L(T)}{1-T},\]
where $A_L(T)$ denotes the normalized Alexander polynomial of the link. 
\end{corollary}

As a second corollary, we combine Theorem \ref{ingel} with results of Gorsky and Némethi \cite{GorskyNemethi} to obtain an interpretation of the motivic integrals $\int_{\Sym_n \CX_f^0} |\omega^{gel}_n|$ in terms of the knot Floer homology $HFL^-(L,v)$ of the link $L$ defined by the plane curve singularity $(f,0)$ \cite{OZ,gridhomologybook}. 

\begin{corollary}[\ref{thm:motivichilb}]\label{introcor}

We have an equality
$$P_{gel}(T)=\sum_{v\in \BZ^r, d\in \BZ} \dim HFL^-_d(L,v) \BL^{d-2h(\underline{v})} T^{|v|},$$
where $h(\underline{v})=\codim_{k[[x,y]]}J_{\underline{v}}$ as in Section \ref{mps}.
%That is, the ''motivic Poincar\'e series" $P_{gel}(T)$ of the one--generator loci of the Hilbert schemes of $f$ introduced in Subsection \ref{mps} computes a specialization of the knot Floer homology of $L$.
\end{corollary}

It seems tempting to try to relate Corollary \ref{introcor}  to the arc-Floer conjecture of Budur-de Bobadilla-L\^e-Nguyen \cite{BdBLN22}, now a theorem for plane curve singularities \cite{arcfloer}, predicting an isomorphism between the compactly supported cohomology of the contact loci $X^0_{f,n}$ and the fixed point Floer homology of the $n$-th iterate of the monodromy of $F_0$. Although we do not know any direct relation between $P_{gel}(T)$ and the geometry of the spaces $X^0_{f,n}$ at the moment, our second main result shows that a closely related integral on $\Sym_n \CX_f^0$ is determined by the motivic Igusa zeta function \cite{DL98}
\[Z_f(T) = \sum_{n \geq 1} [X^0_{f,n}]\BL^{-mn} T^n.\]

\begin{theorem}[\ref{plesym}] \label{inplesym}For any non-constant $f \in k[x_1,\dots,x_m]$ there exists a natural orbifold form $ |\omega_{orb}|^{1/2}$ on $\Sym_n \CX^0_{f}$ such that
\[\sum_{n\geq 0} \int_{\Sym_n \CX^0_{f}} |\omega_{orb}|^{1/2} T^n =  \Exp\left( \BL^{-\frac{m-1}{2}} Z_f(\BL^{-\frac{m-3}{2}} T)\right),\]

where $\Exp: T\calc_k[\BL^{1/2}][[T]] \to 1 + T\calc_k[\BL^{1/2}][[T]]$ denotes the plethystic exponential. 
\end{theorem}

For $m=3$ the relative Hilbert scheme $\Hilb_n(\CX_f)$ provides a crepant resolution of the generic fiber of $\Sym_n \CX^0_{f}$ and a formula similar to Theorem \ref{inplesym} has appeared in \cite{Pa24}, see also Remark \ref{m3}.

Finally, for $m=2$, we can relate the two volume forms $|\omega^{gel}_n|$ and $|\omega_{orb}|^{1/2}$ on $\Sym_n \CX^0_{f}$ by means of the formula 
\[ \BL^{-\ord_{\Delta_f}/2} |\omega_{orb}|^{1/2} = |\omega^{gel}_n|,\] 
where $\Delta_f \subset \Sym_n \CX_{f}$ denotes the discriminant, see Proposition  \ref{orbgel}. In particular, due to the factor $\BL^{-\ord_{\Delta_f}/2}$, the series $P_{gel}(T)$ does not seem to be determined by the classes $[X^0_{f,n}]$. Instead, there is a natural $1$-parameter family 

\[  Q_f(s,T) = \sum_{n\geq 0} \int_{\Sym_n \CX^0_{f}} |\Delta_f|^s |\omega_{orb}|^{1/2} T^n  \in \calc_k[\BL^{1/2}][[T,\BL^{-s}]], \]
for some formal parameter $s$, such that 
\begin{equation}\label{eq:interpolate}
  Q_f(1/2,T) =   P_{gel}(T) \ \text{  and  } \  Q_f(0,T) =  \Exp\left( \BL^{-\frac{1}{2}} Z_f(\BL^{\frac{1}{2}} T)\right).\end{equation}

%In light of the motivic monodromy conjecture \cite{DL98}, which relates the poles of $Z_f(T)$ to the roots of $\zeta_{f,0}(T)$, it might be interesting to get a better understanding of $ Q_f(s,T)$.

A good understanding of $Q_f(s,T)$ might give a new way of relating $Z_f(T)$ to $\zeta_{f,0}(T)$, however, even for a smooth curve computing $Q_f(s,T) $ is equivalent to computing the Igusa zeta functions of all classical discriminant polynomials, see Section \ref{cep}. The latter is an open problem as far as we know, and we hope to come back to it in future work.

The paper is organized as follows: In Section \ref{back} we recall the necessary background on Hilbert schemes and motivic integration. In particular, we discuss Haiman's charts of the Hilbert scheme of $\BA^2$ which we use in Section \ref{gelin} to construct the definable form $|\omega_n^{gel}|$ and prove Theorem \ref{inplesym}. In Section \ref{orsec} we discuss the orbifold formalism for symmetric powers, establish Theorem \ref{inplesym} and compare the orbifold form $|\omega_{orb}|^{1/2}$ with $|\omega_n^{gel}|$. In Section \ref{rkt} we discuss the relation with (symplectic) knot invariants, in particular Corollary \ref{introcor}, and speculate about possible refinements and extensions of our results. Finally, in Section \ref{cep} we discuss explicit examples of the series $Q_f(s,T)$ and its specializations. \\

\textbf{Acknowledgements:} We warmly thank  Javier de la Bodega-Aldama, Arthur Forey, Paolo Ghiggini, François Loeser, Johannes Nicaise, Simon Pepin Lehalleur and Ilaria Rossinelli for interesting discussions around the subjects of this paper.We are very grateful to Eugene Gorsky for helping us with the proof of Proposition~\ref{prop:motivicequalsgel}. 

 A.O. would like to thank EPFL, Lausanne for hospitality during the visits in September 2023 and January  2024, many ideas in the paper were developed during these stays. O.K and D.W. were supported the by Swiss National Science Foundation [no. 196960]. O.K. was also supported by a V\"ais\"al\"a project grant of the Finnish Academy of Science and Letters. A.O   was supported by National Science Foundation grant  DMS-2200798.

\section{Background}\label{back}
\subsection{Hilbert Schemes}
Throughout this section we work over an algebraically closed field $k$ of characteristic $0$.
\subsubsection{Hilbert scheme of the plane}
\label{ha2}

The Hilbert scheme of the plane $\Hilb_n(\BA^2)$ is a smooth and irreducible variety of dimension $2n$. In the work of Haiman \cite{Haiman02} a cover of $\Hilb_n(\BA^2)$ by  affine charts is constructed as follows:

For any subset $M \subset \BN^2$ of size $n$ we may consider the open
\[U_M = \{ I \subset k[X,Y] \ |\ \spn\langle X^pY^q\rangle_{(p,q) \in M} \xrightarrow{\sim} k[X,Y]/I\}.  \]

Any monomial $X^pY^q$ defines a section of the tautological bundle $\CT \to \Hilb_n(\BA^2)$ and the collection $\{ X^pY^q\}_{(p,q)\in M}$ trivializes $\CT_{|U_M}$. Any $f \in k[X,Y]$ defines a global section $f^{[n]}$ of $\CT$ which we can describe on $U_M$ using this trivialization as follows: for $I \in U_M$  we may write the image $\overline{f}$ of $f$ in $k[X,Y]/I$ uniquely as
\[ \overline{f} = \sum_{(p,q)\in M} f_{p,q}^{[n]}(I) X^pY^q.  \]
If we write $f_M^{[n]}:U_M \to \BA^n$ for the function $(f_{p_iq_i}^{[n]})_{1 \leq i \leq n}$ obtained this way we have $f_M^{[n]} = f^{[n]}_{|U_M}$ in the trivialization given by $\{ X^pY^q\}_{(p,q)\in M}$. Here we used the inverse lexicographical order on $\BN^2$,  i.e $(a,b) \leq (c,d)$ iff $b \leq d$ or $b=d$ and $a \leq c$, to order the elements $(p_1,q_1) < \dots < (p_n,q_n)$ of $M$.

Now assume $I \in U_M$ has support (with multiplicities) $(x_1,y_1), \dots ,(x_n,y_n)\in \BA^2$. Then $f_M^{[n]}(I)$ is related to the evaluation of $f$ at the points $(x_i,y_i)$ by the matrix $B_M(x,y) = (x_i^{p_j}y_i^{q_j})_{ij}$ i.e.
\begin{equation}\label{ftof}\begin{pmatrix}
         f(x_1,y_1)\\
         f(x_2,y_2)\\ 
         \vdots \\ 
         f(x_n,y_n) 
     \end{pmatrix}
    =B_M(x,y)
     %\begin{pmatrix}
        % x_1^{p_1}y_1^{q_1} & x_1^{p_2}y_1^{q_2} & \cdots & x_1^{p_n}y_1^{q_n}\\
         %x_2^{p_1}y_2^{q_1} & x_2^{p_2}y_2^{q_2} & \cdots & x_2^{p_n}y_2^{q_n}\\ 
         %\vdots & \vdots & \ddots & \vdots\\ 
      %x_n^{p_1}y_n^{q_1} & x_n^{p_2}y_n^{q_2} & \cdots & x_n^{p_n}y_n^{q_n}
     %\end{pmatrix}
     \begin{pmatrix}
         f_{p_1q_1}^{[n]}(I)\\
         f_{p_2q_2}^{[n]}(I)\\ 
         \vdots \\ 
         f_{p_nq_n}^{[n]}(I) 
     \end{pmatrix} \end{equation}

We write $\Delta_M$ for the for the determinant of $B_M(x,y)$, which is an alternating function on $(\BA^2)^n$. 

In what follows we will consider $M=M_\lambda$ with $M_\lambda$ the Young diagram of a partition $\lambda \vdash n$ and we write $U_\lambda$ for $ U_{M_\lambda}$, $f_{\lambda}^{[n]}$ for $f_{M_\lambda}^{[n]}$ etc. In particular, we have an open covering $\Hilb_n(\BA^2) = \bigcup_{\lambda \vdash n} U_\lambda$.

%The charts \(U_\lambda \) are labeled by the partitions \(\lambda\vdash n\). Below we describe the images
%\(\bar(U)_\lambda\Sym_n(\BA^2)\) of these charts with respect to the Hilbet-Chow map \(h\).

%In \cite{Haiman02} the following functions on \((\BA^2)^n\) play critical role in the construction of the above mentioned charts. 
%\[V(\lambda)(\vec{x},\vec{y})=\det(M_\lambda),\quad (M_\lambda)_{ij}=x_i^{a_j}y_i^{b_j},\]
%where \((a_1,b_1),\dots, (a_n,b_n)\) are the coordinates of the points inside the diagram \(\lambda\):
%\(0\le b_i<\lambda_{a_i}\).

%The preimages \(\widetilde{U}_\lambda=\varphi^{-1}(U_\lambda)\subset (\BA^2)^n\) are defined by the condition

\subsubsection{Hilbert scheme of a plane curve}

Throughout this section we let $f \in k[X,Y]$ be a reduced non-constant polynomial defining a curve $C=\{f=0\} \subset \BA^2$. The Hilbert scheme $\Hilb_n(f)=\Hilb_n(C)$ naturally embeds into $\Hilb_n(\BA^2)$. In fact, we have

\begin{lemma}\label{lcih} Let $\BA^2_S=\Spec(\Sym(\CO^2_S))$ be the affine plane over a base $S$ and $C_S$ the 0-scheme of a section $f\in \Sym(\CO^2_S) $. The Hilbert scheme $\Hilb_n(C_S)$ is the $0$-scheme of the section $f^{[n]} \in H^0(\Hilb_n(\BA^2_S), \CT)$.
\end{lemma}
\begin{proof}
This is a consequence of \cite[Proposition 4]{AIK}.
\end{proof}

For any $k\geq 1$ we write $\Hilb^{\leq k}_n(f) \subset \Hilb_n(f)$ for the open subscheme of ideals $I \subset \CO_C$ that can be generated by $k$ elements. In particular we denote by $\Hilb_n^1(f) \subset \Hilb_n(f)$ the $1$-generator Hilbert scheme i.e. the locus of principal ideals
\[\Hilb_n^1(f) = \{I \in \Hilb_n(f) \ |\ I=(g) \text{ for some } g\in \CO_C\}.\]

\begin{lemma}\label{sml} The $1$-generator subscheme $\Hilb_n^1(f) $ agrees with the smooth locus of $ \Hilb_n(f)$
\end{lemma}
\begin{proof}This is \cite[Proposition 6.5.]{Kass09}. To elaborate, 
the tangent space at $x\in \Hilb^n(f)$ corresponding to a closed subscheme $Z$ is given by $\Hom_{\CO_C}(I_Z,\CO_Z)$. If $x\in \Hilb^n_1(f)$, or equivalently when $Z$ is Cartier, $\Hom_{\CO_C}(I_Z,\CO_Z)\cong \Hom_{\CO_Z}(\CO_C, \CO_Z)$, which has dimension $n$. So these points are in the smooth locus. Conversely, 
if $x\in \Hilb^n(f)\backslash \Hilb_n^1(f)$, we can assume $Z$ is supported at a single point and in this case the completed local ring at that point will be isomorphic to $k[[x,y]]/f$ for some $f\in k[[x,y]]$. Lifting $I_Z$ to $\widetilde{I}_Z\subset k[[x,y]]$, there is an exact sequence 
$$0\to \Hom(I_Z,\CO_Z)\to\Hom(\widetilde{I}_Z,\CO_Z)\to \CO_Z$$ In the proof of \cite[Proposition 6.5.]{Kass09}, it is proved that the image of the last map lies in $\mathfrak{m}_Z$, the maximal ideal of $\CO_Z$. Since $\dim_k \mathfrak{m}_Z=n-1$, the exact sequence implies a lower bound $$\dim_k T_x(\Hilb^n(f))\geq 2n-n+1=n+1$$ and hence the points outside the $1$-generator locus are singular.
\end{proof}

Finally, we consider for any $k\geq 1$ the deformation $f-t^k$ seen as a subscheme of $\BA^2_{k[[t]]}$. 

\begin{lemma}\label{burch} Let  \(k\ge 1\). Then any section $s$ of the relative Hilbert scheme $\Hilb^n(f-t^k)\to \Spec(k[[t]])$ evaluates to an ideal with at most \(k\)-generators: $s_{|\Spec(k)}\in \Hilb_n^{\leq k}(f)$.
%Let $\CO=k[[t]], \mathcal{K}=k((t))$. For a fixed $n\geq 0$, consider the family 
%$\pi: \Hilb^n(f-t^k)\to \Spec(\CO)$, with . 
%Then any section $s$ of $\pi$ evaluates to an ideal with
%at most \(k\)-generators:
%$s_{|\Spec(k)}\in \Hilb_n^{\leq k}(C)$.
%On particular, when $k=1$, $s(0)$ lies in the $1$-generator locus.
\end{lemma}
\begin{proof}
  The local structure ring of the ambient space is  \(R=(k[x,y]\otimes k[t])_{\frm}\) and \(\frm =(x,y,t)\) is the maximal ideal.
  Having a section \(s\) is equivalent to having an ideal \(I\subset R\) such that \(f-t^k\in I\) and \(R/I\) is flat over \(k[[t]]\) of length \(n\).
  Thus \(R/I\) is finite rank over \(k[[t]]\) and since \(\Spec(k[[t]])\) is regular, we conclude that \(R/I\) is Cohen-Macaulay.

  Since \(R/I\) is Cohen-Macaulay and is structure ring of one-dimensional scheme, we have \(\mathrm{depth}(R/I)=1\) and by the Auslander-Buchsbaum theorem \cite[Theorem A 2.15]{Eisenbud05} we conclude that the projective dimension of \(R/I\) is \(2\).
  Hence \(I\) admits two step free resolution and
  by the Hilbert-Burch theorem\cite[Theorem 3.2]{Eisenbud05} there is \(g\) by \(g+1\) matrix \(A\) with entries in \(\frm\) such that we have short exact sequence:
  \[0\to R^{g}\xrightarrow{A} R^{g+1}\xrightarrow{G}I\to 0.\]
  where \(G\)  is the row of the minors of \(A\) and \(g+1\) is the minimal number of
  generators of \(I\).

  Next, we observe that \(f-t^k\in \frm^k\setminus\frm^{k+1}\) and that the entries
  of \(G\) are in \(\frm^g\). Hence we conclude that \(g\le k\) and \(I\) has
  \(g+1\) generators \(h_1,\dots,h_{g+1}\). If \(g<k\) then the statement is proven
  otherwise we need to argue explain why \(f-t^k\) is one of the generators \(h_i\).

  Indeed, if \(g=k\) then  \(h_i=\sum_{j=0}^k t^j\alpha_j^i\), \(\alpha_j^i\in \frm^{n-j}\). On the other hand \(f-t^k=\sum_i \beta_i h_i\), \(\beta_i\in R\) and hence
  at least one of \(\beta_i\) is not vanishing modulo \(\frm\).  In particular,
  \(\beta_i\) is a unit in \(R\) and \(h_i=(f-t^k)\beta_i^{-1}-\sum_{j\ne i}h_i\beta_j/\beta_i\) and \(I\) is generated by \(f-t^k\) and \(h_j\), \(j\ne i\).
 \end{proof}

\subsection{Motivic integration}

\subsubsection{Grothendieck Rings}
For any $k$-variety $Z$ we write $K_0(Var_Z)$ for the relative \linebreak Grothendieck ring of varieties over $Z$ and $\BL$ for the class of $\BA^1 \times Z \to Z$. We further write $\CM_Z$ for the localized ring $K_0(Var_Z)[\BL^{-1}]$ and $\calc_Z =  \CM_Z[(1-\BL^i)^{-1}:i \geq 1]$. If $Z = \Spec(k)$ we simply write $K_0(Var_k)$, $\CM_k$ and $\calc_k$ instead of $K_0(Var_{\Spec(k)})$, $\CM_{\Spec(k)}$ and $\calc_{\Spec(k)}$.

In \cite{GLM04} the authors construct power structures on $K_0(Var_k)$ and $\CM_k$ and in particular plethystic operations. For $\CR$ either  $K_0(Var_k)$ or $\CM_k$, the plethystic exponential is a group homomorphism 
\[ \Exp: T \CR[[T]] \to 1+ T \CR[[T]],  \]
defined by the formula 
\[ \Exp( a_1T +a_2T^2 + \dots) = \prod_{n \geq 1} (1-T^n)^{-a_n}.  \]
Here, if $a = [X] \in \CR$ is the class of a variety, then 
\[ (1-T)^{-a} = \zeta_X(T) = \sum_{m \geq 0} [\Sym_m X] T^m   \]
is the Kapranov zeta function of $X$ \cite{Ka00}, with $\Sym_mX$ the $m$-th symmetric power of $X$ defined by 
\[ \Sym_m X = X \times X \times \dots \times X / S_m, \]
where $S_m$ denotes the $m$-th symmetric group. We extend this by defining the $m$-th symmetric product $\Sym_m(a)$ of any element $a \in \CR$ by 
\[\Exp(aT) =   \sum_{m \geq 0} \Sym_m(a) T^m.\]
By \cite[Statement 3]{GLM04} we have in particular $\Sym_m(\BL^{\pm 1}a) = \BL^{\pm m}\Sym_m(a)$. Following \cite[Appendix B]{DM15} one can further extend $\Exp$ and thus $\Sym_m$ for $m\geq 0$, to the ring
\[\calc_k[\BL^{1/2}] =  \CM_k[\BL^{1/2}, (1-\BL^i)^{-1}:i \geq 1],\]
in such a way that $\Sym_m(\BL^{\pm 1/2}a) = \BL^{\pm m/2}\Sym_m(a)$. Here $\BL^{1/2}$ corresponds to the \textit{negative square root} in \textit{loc. cit.}

\subsubsection{Order function and absolute values}

Let $\CX$ be a $k[[t]]$-scheme and $\CZ \subset \CX$ a closed subscheme with ideal sheaf $\CI$. For any $x \in \CX(k[[t]])$ we define the order $\ord_\CZ(x)$ of $x$ along $\CZ$ as
\[ \ord_\CZ(x) = \begin{cases}d  \ &\text{ if } x^*\CI = (t^d) \\    
\infty  \  &\text{ if }  x^*\CI = (0).\end{cases} \]
Notice that $ \ord_\CZ(x) = \infty$ if and only if $x \in \CZ(k[[t]])$. If $\CX$ is affine and $\CI$ generated by $z_1,\dots,z_k:\CX \to \BA^1$, then $\ord_{\CZ}(x)= \min_{1 \leq i\leq k}\{ \ord_t z_i(x) \}$.

 There is a unique way of extending $\ord_\CZ$ to a function
 \[ \ord_\CZ:\CX(\overline{k[[t]]}) = \bigcup_{r \geq 1} \CX(k[[t^{1/r}]]) \to \BQ \cup \{\infty\}. \]

Given a subscheme $\CY \subset \CX$ and $y \in \CY(\overline{k[[t]]})$ we will use that
\[ \ord_{\CZ \cap \CY}(y) = \ord_{\CZ}(y), \]
where on the right we consider $y$ as an $\overline{k[[t]]}$-point of $\CX$. 

Finally, given a regular map $f:\CX \to \BA^1$ or a subscheme $\CZ \subset \CX$ we write $|f|$ and $|\CZ|$ for the maps $\CX(k[[t]]) \to \CM_k$ defined by
\begin{equation}\label{motmap} x \mapsto \BL^{-\ord_t(f(x))} \text{ and } x \mapsto \BL^{-\ord_\CZ(x)}.
\end{equation}

When working with $k[[t^{1/n}]]$-points for some $n \geq 1$ we will write $|\cdot|_n$ to indicate that we compute the absolute value with respect to the uniformizer $t^{1/n}$.

\subsubsection{Cluckers-Loeser integration}

We use the integration theory developed in \cite{CL08}.  As in \cite[16.2]{CL08} and in the notation of \cite{LoeserWyss19} we consider the category $\Def$ of definable subassignments in the Denef-Pas language restricted to the theory of algebraically closed fields. 

Concretely, an object $h \in \Def$ is given by an assignment 
\[ K \mapsto h(K) \subset k((t))^n \times k^m \times \BZ^r, \]
where $K$ runs through all algebraically closed fields containing $k$, $m,n,r \geq 0$ and $h(K)$ is cut out by a formula in the Denef-Pas language over $k$. In particular any affine finite type scheme $X \subset \BA^n$ over $k((t))$ defines an element in $\Def$, which we still denote by $X$, by means of 
\[ K \mapsto X(K((t))) \subset K((t))^n \times K^0 \times \BZ^0. \]

Similarly, for an affine finite type scheme $\CX \subset \BA^n$ over $k[[t]]$ we still write $\CX$ for  the assignment
\[K \mapsto \CX(K[[t]]).\]

Finally, we will sometimes consider the assignments $X_n, \CX_n$ given by $K \mapsto X(K((t^{1/n})))$ or $K \mapsto \CX(K[[t^{1/n}]])$ for some integer $n \geq 1$. Notice that these are simply the assignments associated to the Weil restrictions of $X \times_{k((t))} k((t^{1/n}))$ and $X \times_{k[[t]]} k[[t^{1/n}]]$ along $k((t)) \to k((t^{1/n}))$ and $k[[t]] \to k[[t^{1/n}]]$ respectively.

These constructions can be globalized to obtain a category $\GDef$ consisting of definable subassignments of algebraic varieties. 

To any $S \in \Def$ one can associate a ring of constructible motivic functions $\calc(S)$ in such a way that for the terminal object $*_k \in \Def$ we have
\[ \calc(*_k) = \calc_k =  K_0(Var_k)[\BL^{-1}, (1-\BL^i)^{-1}:i \geq 1].\]
Typical examples of constructible motivic functions are the maps in \eqref{motmap}.

The content of \cite{CL08} is an integration formalism for constructible motivic functions with respect to definable volume forms satisfying in particular a change of variables formula and a Fubini theorem.

The forms one integrates in this theory are definable forms $|\omega|$ as defined in \cite[Section 8.2]{CL08}. On a definable subassignment associated with a $k((t))$-variety $X$ any top-dimensional form $\omega$ on the smooth locus $X^{sm}$ of $X$ defines a definable form $|\omega|$ and we have an identification $|\omega| = |\omega'|$ if  $\omega = f \omega'$ for some $f: X^{sm} \to \BA^1$ with $|f| \equiv 1$.  

The definable forms in this paper will thus typically be constructed as follows. Given $X/k((t))$, an open cover $X^{sm} = \bigcup_i U_i$ of its smooth locus and forms $\omega_i$ on $U_i$ such than on $U_i \cap U_j$ we have $\omega_i = f_{ij} \omega_j$ for a function $f_{ij}$ satisfying $|f_{ij}| \equiv 1$, then the $U_i$ glue together to a definable form on $X$.

Given a function $f \in \calc(X)$, integrable with respect to a definable form $|\omega|$, one obtains a class

\[\int_{X} f |\omega| \in \calc_k.\]

In the following if we write $\int_{X} f |\omega| = a$ we mean $f$ is integrable with respect to $|\omega|$ and $\int_{X} f |\omega| = a$.

One way to compute motivic integrals is by means of Néron smoothenings. %following \cite{Chambert-LoirNicaiseSebag18}.
Given $\CX /k[[t]]$ separated, flat and of finite type with smooth generic fiber, there exists a smooth finite type $k[[t]]$-scheme $\CX$ and a $k[[t]]$-morphism $\CY \to \CX$, which is generically an isomorphism and for every extension $K/k$ induces a bijection $\CY(K[[t]]) \xrightarrow{\sim}  \CX(K[[t]])$ \cite[Theorem 3.1.3]{BoschLutkebohmertRaynaud90}. For any volume form $\omega$ on $\CX_{k((t))}$ one then has \cite[Proposition 12.6]{cluckers2015motivic}

\begin{equation}\label{weil} \int_{\CX} |\omega| = \BL^{-\dim \CX} \sum_{C \in \pi_0(\CX_k)}\BL^{-\ord_C \omega}[C]. \end{equation}

%An important special case of this construction is the case when $X$ is the generic fiber of $\CX \to \Spec(k[[t]])$ a smooth, equidimensional, separated finite type morphism. Here we may chose an open covering of all of $\CX$ and then the condition $|f_{ij}| \equiv 1$ is automatically satisfied. We call the resulting form $|\omega_{Weil}|$. 

%\begin{lemma}\label{weil} Let $\CX \to \Spec(k[[t]])$ be a smooth, equidimensional, separated finite type morphism. Then we have
%\[\int_{\CX} |\omega_{Weil}| = \BL^{-\dim \CX} [\CX_k] \in \calc_k. \]
%\end{lemma}

\subsubsection{Motivic Igusa Zeta functions}\label{mizf}

Consider a non-constant polynomial map $f:\BA^m \to \BA^1$ over $k$ with $f(0)=0$. The  (local) $n$-th restricted contact locus of $f$ is defined as
\begin{equation}
\label{def:contactlocus}
X^0_{f,n} = \{ x \in \BA^m(k[t]/t^{n+1})_0 \ |\ f(x) = t^n \in k[t]/t^{n+1} \},
\end{equation}
where $\BA^m(k[t]/t^{n+1})_0 $ denotes the locus of jets reducing to $0$ modulo $t$.

The \textit{motivic Igusa zeta function of $f$} is defined as the generating series 
\begin{equation}\label{eq:IgusaZ}
Z_f(T) = \sum_{n \geq 1} [X^0_{f,n}]\BL^{-mn} T^n \in \CM_k[[T]]. \end{equation}
In fact, using an embedded resolution of singularities one can show that $Z_f(T) \in \calc_{k}(T)$  \cite{DL98}. There is an additional $\mu_n$-action on each $X^0_{f,n}$ but we will ignore it in this paper. 

Our goal is to give an expression for the plethystic exponential $\Exp(Z_f(T)) \in \calc_{k}[[T]]$. For this we'll need a slightly different way of writing the Igusa zeta function following \cite{NS07}.

Let $\omega_m=dx_1\wedge \dots \wedge dx_m$ be the standard form on $\BA^m$ and consider the relative $m-1$-form $\omega_m /df$ on the smooth locus of $f$. 
%\[V = \{ x \in \BA^m \ |\ df(x) \neq 0\}, \]

%For this assume that $f$ is smooth on  $V=X \setminus X_s$ and that $X$ admits a global non-vanishing $m$-form $\omega$. We denote by $\omega /df$  the relative $m-1$-form on $V \to \BA^1 \setminus \{0\}$. 

Next consider the $k[[t]]$-scheme $\CX_f$ obtained as the pullback of $\BA^m \xrightarrow{f} \BA^1$ under the completion morphism $\Spec(k[[t]]) \to \Spec(k[t])$, explicitly
\[ \CX_f = \Spec(k[[t]][x_1,\dots,x_m] / f-t). \]
Then $\omega_m /df$ induces a volume form $\omega_f$ on the generic fiber of $\CX_f$ called the \textit{Gelfand-Leray form}.

For any $n \geq 1$ we write $\CX^0_{f,n}$ for the definable subassignment whose $K$-points are given by $\{ x \in \CX_f(K[[t^{1/n}]])\ | \ x_{K} = 0  \}$. 
\begin{lemma}\cite[Lemma 9.9]{NS07}\label{clin} For any $n\geq 1$ we have
\[ \int_{\CX^0_{f,n}} |\omega_f|_n = \BL^{-(n+1)(m-1)} [X^0_{f,n}] \in \calc_k.\]
\end{lemma}

The proof of Lemma \ref{clin} uses an embedded resolution of singularities $f$ to compute both sides independently. Technically a different integration theory is used in \textit{loc. cit.}, but the same argument can be applied in the Cluckers-Loeser theory using \eqref{weil}.

\section{Arcs and Gelfand forms on Hilbert schemes}\label{gelin}

Let $f\in k[X,Y]$ be reduced and nonconstant with an isolated singularity at $0$. We write $C=\{f=0\} \subset \BA^2$ for the associated curve. Consider the relative curve $\CX_f /k[[t]]$ as in Section \ref{mizf} and the relative Hilbert-Chow morphism
\[h: \Hilb_n(\CX_f) \to \Sym_n(\CX_f). \]
Since the generic fiber of $\CX_f$ is smooth, $h_{|k((t))}$ is an isomorphism. Furthermore $h$ is proper and thus we have for every $K/k$ and every $K[[t]]$-point of $\Sym_n(\CX_f)$ a unique lift to a $K[[t]]$-point of $\Hilb_n(\CX_f)$, in other words the definable subassignments associated with $ \Sym_n(\CX_f)$ and $\Hilb_n(\CX_f)$ are isomorphic. 

Recall that $\Hilb_n^1(f) \subset \Hilb_n(f)$ denotes the $1$-generator locus. We further write $\Hilb_n^1(f)_0 =  \Hilb_n^1(f)\cap h^{-1}(n[0])$ for the closed subscheme of ideals supported at the origin. As a direct consequence of Lemma \ref{burch} we have

\begin{corollary}\label{hbthm} For any $K/k$ and any $x \in \Hilb_n(\CX_f)(K[[t]])$ such that $(h_*x)_{|K} = n[0] \in \Sym_n(C)$ we have $x_{|k} \in \Hilb_n^1(f)_0$.
\end{corollary}

Next we use Corollary \ref{hbthm} to construct a definable volume form on $\Sym_n(\CX_f^0)$, the definable subassignment  given by
\[K \mapsto \{ x \in \Sym_n(\CX_f)(K[[t]]) \ |\ x_K = n[0]   \}.\]
First, by Corollary \ref{hbthm} we have an isomorphism of $\Sym_n(\CX_f^0)$ with $\Hilb_n(\CX_f^0) $ defined by
\[K \mapsto \{ x \in \Hilb_n(\CX_f)(K[[t]]) \ |\ x_{|K} \in \Hilb_n^1(C)_0 \}.\] 
Next recall the open covering $\Hilb_n(\BA^2) = \bigcup_{\lambda \vdash n} U_\lambda$ from Section \ref{ha2}. Since $\Hilb_n(\CX_f) \subset \Hilb_n(\BA^2)$ we obtain an open covering of $\Hilb_n(\CX_f)$ by setting $U_{f,\lambda} = U_\lambda \cap \Hilb_n(\CX_f)$. By Lemma \ref{lcih} we can also describe $U_{f,\lambda}$ as
\[ U_{f,\lambda} = \{f^{[n]}_\lambda = 0\} \subset U_\lambda. \]
Now since $\Hilb_n(\BA^2)$ is algebraic symplectic there exists in particular a global non-vanishing top-form $\omega^{[n]}$. We define the \textit{Gelfand form} $\omega_\lambda^{gel}$ on the smooth locus $U_{f,\lambda}^{sm}\subset U_{f,\lambda}$ as
\[ \omega_\lambda^{gel} = \omega^{[n]}/ f_\lambda^{[n]*}\nu, \]
where $\nu=dy_1\wedge \dots \wedge dy_n$ is the standard form on $\BA^n$.

\begin{proposition} The definable forms $| \omega_\lambda^{gel}|$ glue together to a give a non-vanishing definable form $|\omega^{gel}_n|$ on $\Sym_n(\CX_f^0) \cong \Hilb_n(\CX_f^0)$
\end{proposition}
\begin{proof} For any two partitions $\lambda,\mu \vdash n$ the two forms $\omega_\lambda^{gel}, \omega_\mu^{gel}$ differ by an invertible function $u_{\lambda \mu}$ on $U_{f,\lambda}^{sm} \cap U_{f,\mu}^{sm}$.  By combining Lemma \ref{sml} and Corollary \ref{hbthm} we see that for any $K/k$, any $x\in \Hilb_n(\CX_f^0)(K[[t]])$ factors through the smooth locus of $\Hilb_n(\CX_f)$ and thus $|u_{\lambda \mu}(x)| = 1$ which implies the proposition.
\end{proof}

\begin{theorem}\label{gelhil} For any $n\geq 1$ we have
\[ \int_{\Sym_n \CX_f^0} |\omega^{gel}_n| = \BL^{-n} [\Hilb_n^1(f)_0]. \]
\end{theorem}
\begin{proof}
Since $\Sym_n(\CX_f^0) \cong \Hilb_n(\CX_f^0)$ this follows from  \eqref{weil}. 
\end{proof}

We also record the following for later.

\begin{lemma}\label{hev} Let $\CI \in U_\lambda(K[[t]])$ and $(x,y) \in (\BA^2)^n(\overline{K((t))})$ a preimage of $\supp(\CI_{|K((t))})\in \Sym_n(\BA^2)(K((t)))$ under the Hilbert-Chow morphism. Then 
\[ \ord_t(\Delta_\lambda(x,y)) \leq \ord_t(\Delta_M(x,y)),\]
for all $M \subset \BN^2$ with $|M| = n$. 
\end{lemma}
\begin{proof}
This follows since on $U_\lambda$ the fraction $\Delta_M /\Delta_\lambda$ defines a regular function\cite[Proposition 2.6.]{Haimantq}, thus $\Delta_M(x,y) / \Delta_\lambda(x,y) \in K[[t]]$. 
\end{proof}

\subsection{The motivic Poincar\'e series}\label{mps}
In this subsection we assume $k= \CC$. Consider the generating series 
\[P_{gel}(T) = \sum_{n \geq 0} \int_{\Sym_n \CX_f^0} |\omega^{gel}_n| T^n   \in \calc_\CC[[T]] .\]
We will show that $P_{gel}(T)$ equals the motivic Poincar\'e series of $f$ introduced in \cite{CDGZ04}.

The motivic Poincar\'e series of \cite{CDGZ04} is defined as follows when $m=2$. If the germ $(C,0)$ defined by $f$ contains $r$ branches, we can define a multi--index filtration on $k[[x,y]]$ by setting $$J_{\underline{v}}=
\{g\in k[[x,y]]|v_i(g)\geq v_i\}$$ where $\underline{v}=(v_1,\ldots,v_r)\in \BZ^r$  and $v_i(g)$ is the order of the pullback of $f$ to the normalization of the $i$-th branch of $C$. Let $h(\underline{v})=\codim_{k[[x,y]]}J_{\underline{v}}$ and write
$$L_f(t_1,\ldots,t_r,\BL^{-1})=\sum_{\underline{v}\in \BZ^r}\frac{\BL^{-h(\underline{v})}-\BL^{-h(\underline{v}+\underline{1})}}{1-\BL^{-1}}t^{\underline{v}}$$ By \cite{CDGZ04} one can define the motivic Poincar\'e series of $f$ as $$P_f(t_1,\ldots,t_r;\BL)=\frac{L_f(t_1,\ldots,t_r,\BL^{-1})\prod_{i=1}^r(t_i-1)}{t_1\cdots t_r-1}$$ In particular, when $r=1$ and the curve is unibranch, we have $L_f(t,\BL^{-1})=P_f(t,\BL^{-1})$. 

Our interest in the motivic Poincar\'e series is encapsulated in
\begin{proposition}
\label{prop:motivicequalsgel}
 $$P_f(t_1=\dots=t_r=T,\BL^{-1})=\sum_{n=1}^\infty \BL^{-n} [\Hilb_n^1(f)_0] T^n=P_{gel}(T).$$
\end{proposition}
\begin{proof}
As explained in \cite[Section 3]{GorskyNemethi}, the coefficient of $t^v=t_1^{v_1}\cdots t_r^{v_r}$ in $P_f(t_1,\ldots,t_r,\BL^{-1})$ is given by \begin{equation}\label{eq:coeffs}\pi_v(\BL^{-1})=\BL^{1-h(v+\underline{1})}[\mathbb{P}\mathcal{H}'(v)],\end{equation} where $[\mathbb{P}\mathcal{H}'(v)]$ is the motive of the projectivization of the hyperplane arrangement $\mathcal{H}'(v)$ defined by
$$\mathcal{H}'(v):=\left.\frac{J_v}{J_{v+\underline{1}}}\middle\backslash\bigcup_i \frac{J_{v+e_i}}{J_{v+\underline{1}}}\right.$$ where $e_i$ is the $i$-th basis vector.

Next, note that any nonzero function not proportional to a factor of $f$ gives a finite codimension ideal in $\CO_C=k[[x,y]]/f$. Note that these are exactly the functions with finite valuation on each component of $C$. We can call these regular elements and denote them $\CO_C^{reg}$. For any $g\in \CO_C^{reg}$ one has 
\begin{equation}
\label{eq:codim}
\dim \CO_C/(g)=\sum_{i=1}^r v_i(g),
\end{equation} 
which follows from the additivity of intersection numbers of plane curves or from the following direct argument: 
Let \(\CO_{C_i}\) be the local rings of the branches and the projections \(\pi_i:\CO_C\to \CO_{C_i}\)
yield an embedding \(\pi:\CO_{C}\to \oplus_i\CO_{C_i}\). The cokernel of the embedding is of finite length and \(\pi(g)\) is non-zero divisor for \(g\in \CO_C^{reg}\) hence
we apply \cite[Proposition 6]{OblomkovShende12} for \(A=\CO_C\) and \(B=\CO_{C_i}\) and \(f=g\). Thus showed that \(\dim(\CO_C/(g))=\sum_i\dim(\CO_{C_i}/(\pi_i(g))).\)
Finally, \(\pi_i(g)\) is not  a zero-divisor hence
the equality \(\dim(\CO_{C_i}/(\pi_i(g)))=v_i(g)\) follows from applying  \cite[Proposition 6]{OblomkovShende12} to the inclusion
\(\CO_{C_i}\to k[[t]]\).

Now there is a map $\CO_C^{reg}\to \bigsqcup_n \Hilb_n^1(f)_0$ 
sending $g\mapsto I=(g)$ and the only ambiguity in defining this principal ideal is 
multiplication by invertible functions in $\CO_C$, which gives a constructible
bijection $(J_v\backslash \bigcup_i J_{v+e_i})/\CO_C^{*}\xrightarrow{\sim} \Hilb^1_v(f)_0$, where  $\Hilb^1_v(f)_0$ is the subset of $\Hilb_n^1(f)_0$ where the 
valuations on the components correspond to $v$.

On the other hand there is a fibration
\[  (J_v\backslash \bigcup_i J_{v+e_i})/ \CO_C^{*} \to  \mathcal{H}'(v) \]
whose fiber over $[g] \in \mathcal{H}'(v)$ can be identified with $g+J_{v+\underline{1}}/1 + \mathfrak{m}_C$, where  $ \mathfrak{m}_C \subset \CO_C$ denotes the maximal ideal. 
%want to show that it is an affine space of dimension $|v|+1-h(v+\underline{1})$. 

% We claim that $g+J_{v+\underline{1}}/1 + \mathfrak{m}_C$ is an affine space whose dimension is given by $|v|+1-h(v+\underline{1})$. First of all, $(1 + \mathfrak{m}_C)$ is a prounipotent algebraic group acting on an infinite-dimensional affine space by translations:
%$$(1 + \mathfrak{m}_C)(g+J_{v+\underline{1}})=g+J_{v+\underline{1}}+g\mathfrak{m}_C+\mathfrak{m}_CJ_{v+\underline{1}}$$ \todo{how to get rid of the extra term?}

 Let us show that the quotient $g+J_{v+\underline{1}}/1 + \mathfrak{m}_C$ has the class \(\mathbb{L}^{d(v)}\)  where \(d(v)\)  is the dimension of \(J_{v+\underline{1}}/g\mathfrak{m}_C\). Indeed, the linear
space \(\mathfrak{m}_CJ_{v+\underline{1}}\) acts by shifts on the affine space \(g+J_{v+\underline{1}}\) and for any $k\geq 0$ there is a well-defined projection map:
\[\pi_k: \left((g+J_{v+\underline{1}})/\mathfrak{m}^{k+1}_CJ_{v+\underline{1}})\right)/(1+\mathfrak{m}_C)\to  \left((g+J_{v+\underline{1}}/\mathfrak{m}^k_CJ_{v+\underline{1}}\right)/(1+\mathfrak{m}_C).\]

The fibers of the map \(\pi_k\)  are  isomorphic to the affine space
\[V_k=(\mathfrak{m}^k_CJ_{v+\underline{1}}/\mathfrak{m}^{k+1}_CJ_{v+\underline{1}})/(g\mathfrak{m}_C\cap \mathfrak{m}_C^kJ_{v+\underline{1}})\] by the following argument.
Let us  pick \(\bar{h}\in J_{v+\underline{1}}/\mathfrak{m}_C^kJ_{v+\underline{1}}\) and \(h,h' \in J_{v+\underline{1}}/\mathfrak{m}_C^{k+1}J_{v+\underline{1}}\ \) are two lifts of
\(\bar{h}\). Then \(g+h\) and \(g+h'\) are in the same \(1+\mathfrak{m}_C\)-orbit if and only \(h-h'=g m\), \(m\in \mathfrak{m}_C\), and the statement follows.
To complete the argument we observe that \(V_k\) is a point for large \(k\) and \(d(v)=\sum_{k\ge 0}\dim(V_k)\).

To finish the proof we observe that $d(v) = \dim J_{v+\underline{1}}/g\mathfrak{m}_C  = |v|+1 -h(v+\underline{1})$, which follows from the short exact sequence
\[ 0 \to J_{v+\underline{1}}/g\mathfrak{m}_C \to \CO_C/g\mathfrak{m}_C \to \CO_C/J_{v+\underline{1}} \to  0,  \]
since $ \dim \CO_C/J_{v+\underline{1}} = h(v+\underline{1})$ by definition and and $\dim \CO_C/g\mathfrak{m}_C  = |v|+1$ by \eqref{eq:codim}.
%Now let us compute \(d(v)\). If $h\in \mathfrak{m}_C$ then $v_i(gh)>v_i(g)$ for all $i$ so $gh\in J_{v+\underline{1}}$ and $g\mathfrak{m}_C\subseteq J_{v+\underline{1}}$. So the quotient by the above action is an affine space isomorphic to 
%$J_{v+\underline{1}}/(g\mathfrak{m}_C+\mathfrak{m}_CJ_{v+\underline{1}})$. To compute the dimension, observe
%$$d(v)=\dim J_{v+\underline{1}}/g\mathfrak{m}_C=\dim \CO_C/g\mathfrak{m}_C-\dim \CO_C/J_{v+\underline{1}}.$$ But by Eq. \eqref{eq:codim} the first term is $|v|+1$ and the second term is $-h(v+\underline{1})$.

%Now let $C$ be unibranch and write $\Gamma \subset \BN$ for its semi-group and $H = \BN \setminus \Gamma$ for its complement. As in \cite[Lemma 25]{OblomkovShende12} we can eliminate coefficients iteratively to get that $g+J_{v+1}/1 + \mathfrak{m}_C$ is an affine space whose dimension is given by 
%$|(v+H) \cap \Gamma|$. Now there is a bijection 
%\[(v+H) \cap \Gamma \xrightarrow{\sim} \{ h \in H\ |\ h <v\},\]
%given by sending $v+h$ to $h-kv$ with $k$ such that $0<h-kv<v$. The proof now follows from the observation that $| \{ h \in H\ |\ h <v\}| = v+1-h(v+1)$. 
%On the other hand, $$\left[\left(\frac{J_v}{J_{v+\underline{1}}}\middle\backslash \bigcup_i \frac{J_{v+e_i}}{J_{v+\underline{1}}}\right)/k^*\right]\cong \left[(J_v\backslash \bigcup_i J_{v+e_i})/ \CO_C^{*}\right].$$ Upon specialization of all $t_1=\cdots=t_r=T$, using Eqs. \eqref{eq:coeffs}, \eqref{eq:codim}, we have that the coefficient of $T^n$ in $P_f(t_1=\cdots=t_r=T,\BL^{-1})$ is $\BL^{-n}[\Hilb^1_n(f)_0]$.
\end{proof}
\subsection{The Euler characteristic specialization}
\label{sec:specialization}
In this subsection we assume $k= \CC$. Consider the generating series 
\[P_{gel}(T) = \sum_{n \geq 0} \int_{\Sym_n \CX_f^0} |\omega^{gel}_n| T^n   \in \calc_\CC[[T]] .\]

By Theorem \ref{gelhil} we have $P_{gel}(T) = \sum_{n \geq 0}  \BL^{-n} [\Hilb_n^1(f)_0] T^n$ and thus $P_{gel}(T)$ lies in the image of $\CM_\CC[[T]] \to \calc_\CC[[T]]$. In particular, we may consider its Euler-characteristic specialization
\[ P_{gel,\chi}(T) = \sum_{n\geq 0} \chi(\Hilb_n^1(f)_0)T^n. \]
If $L$ denotes the link defined by the plane curve singularity $(f,0)$ and $A_L(T)  \in 1 +T\BZ[T]$ the normalized Alexander polynomial, it follows from \cite[Section 3]{OblomkovShende12} that 
\[ P_{gel,\chi}(T) = \frac{A_L(T)}{1-T}.\] 

On the other hand, since $|\omega_n^{gel}|$ is non-vanishing and the generic fiber of $\CX_f^0$ is smooth, the Euler-characteristic specialization of $\int_{\Sym_n \CX_f^0}|\omega_n^{gel}|$ equals the Euler-characteristic specialization of the motivic Serre invariant of $\Sym_n \CX_f^0$ \cite[Section 6.1]{NS07}. Using this, we deduce from the trace formula \cite[Theorem 5.4]{NS07} 
\[ \chi\left( \int_{\Sym_n \CX_f^0}|\omega_n^{gel}|\right) = Tr(\phi, H_B^*(\Sym_n(\CX_f^0),\BQ_\ell)) = \sum_{i \geq 0}(-1)^i Tr(\phi, H_B^i(\Sym_n(\CX_f^0),\BQ_\ell),  \]
where $H_B^*(\Sym_n(\CX_f^0),\BQ_\ell)$ denotes the Berkovich étale cohomology of the rigid analytic space associated to $\Sym_n(\CX_f^0)$ and $\phi$ is the topological generator of $\Gal(\CC((t))) \cong \hat{\mu}$ given by $t^{1/n} \mapsto \exp(2\pi i /n)t^{1/n}$. 

\begin{lemma}\label{nsym} We have 
\[H_B^*(\Sym_n(\CX_f^0),\BQ_\ell)) \cong  (H_B^*(\CX_f^0,\BQ_\ell)^{\otimes n})^{S_n},  \]
as graded $\hat{\mu}$-modules.
\end{lemma}

\begin{proof} Instead of working directly with Berkovich étale cohomology we use the comparison with $\ell$-adic nearby cycles \cite[Corollary 3.5]{Be96} implying 
\begin{align*}H_B^*(\Sym_n(\CX_f^0),\BQ_{\ell,\Sym_n(\CX_f)}) &\cong H^*(R\psi_\eta(\BQ_{\ell,\Sym_n(\CX_f)}) )_{|n[0]}), \\
H_B^*(\CX_f^0,\BQ_{\ell,\CX_f}) &\cong H^*(R\psi_\eta(\BQ_{\ell,\CX_f})_{|0}). \end{align*}
Here we wrote $\BQ_{\ell,\Sym_n(\CX_f)}$ and $\BQ_{\ell,\CX_f}$ for the constant sheaves on (the generic fibers of) $\Sym_n(\CX_f)$ and $\CX_f$. Therefor it suffices to show that we have a quasi-isomorphism
\[R\psi_\eta(\BQ_{\ell,\Sym_n(\CX_f)}) \cong (\pi_*\boxtimes^n R\psi_\eta(\BQ_{\ell,\CX_f}))^{S_n}   \]
 in $D^b_c(\Sym_n(\CX_f)_{|k})$, where $\boxtimes$ denotes the exterior product on $\CX_f^n$. Notice that taking $S_n$-invariants can be defined via a projector as for example in \cite{MS12} and thus $R\psi_\eta$ commutes with taking $S_n$-invariants. Furthermore $R\psi_\eta$ commutes with proper pushforward \cite[XIII 2.1.7]{DK06} and exterior products \cite[Lemma 5.1.1]{BB93}. The lemma thus follows from the identity $\BQ_{\ell,\Sym_n(\CX_f)} \cong (\pi_*\boxtimes^n\BQ_{\ell,\CX_f})^{S_n}$.
\end{proof}

Combining Lemma \ref{nsym} with Lemma \ref{trexp} below we deduce
\[ P_{gel,\chi}(T) = \exp\left(\sum_{n\geq 1} \frac{Tr(\phi^n,H^*(\CX_f^0,\BQ_\ell))}{n} T^n\right).  \]

Finally choosing an embedding $\BQ_\ell \subset \CC$, we deduce from \cite[Theorem 9.2]{NS07} the equality
\begin{equation}
\label{eq:monodromyzetafn}
\exp\left(\sum_{n\geq 1} \frac{Tr(\phi^n,H^*(\CX_f^0,\BQ_\ell))}{n} T^n\right) = \zeta_{f,0}(T),
\end{equation}
where $\zeta_{f,0}(T)$ denotes the monodromy zeta function of the Milnor fiber of $f$ at $0$. Putting all of this together we find

\begin{corollary}\label{miln} For any reduced plance curve singularity $(f,0)$ with link $L$ we have 
\[ \zeta_{f,0}(T) = \frac{A_L(T)}{1-T}.\]
\end{corollary}

Corollary \ref{miln} is originally due to Milnor \cite[Lemma 10.1]{Mi16}.  

\begin{remark}In \cite{gusein2002integrals} the authors consider the series $ P_{gel,\chi}(T)$ as the Poincaré series of the ring of functions on $(f,0)$. Our proof of Corollary \ref{miln} can be seen as an answer to their question, whether there is a direct link between the Poincaré series and the monodromy zeta function. \cite[Section 1]{gusein2002integrals}.
\end{remark}

\begin{lemma}\label{trexp} Let $k$ be a field of characteristic $0$, $V^*$ a finite dimensional $\BZ$-graded $k$-vector space and $\phi^*$ an endomorphism of $V^*$ preserving the grading. Then

\[\exp\left(\sum_{n\geq 1} \frac{Tr((\phi^*)^n,V^*)}{n} T^n\right) = \sum_{n\geq 0} Tr((\phi^*)^{\otimes n}, ((V^*)^{\otimes n})^{S_n})T^n = \prod_{i\in \BZ} \det(1-\phi^i T)^{(-1)^{i+1}},\]
where for a graded endomorphism $\psi^*$ of a finite dimensional graded $k$-vector space $W^*$ we wrote $Tr(\psi^*, W^*) = \sum_{i \in \BZ} (-1)^i Tr(\psi^i,W^i)$. 

%\todo{I think it's simply $1/\det(1-\phi T)$. Should there be something about the grading here? Yes, I clarified the statement}
\end{lemma}
\begin{proof} 
When there is no grading and $\phi$ is diagonalizable with eigenvalues $\lambda_1,\ldots, \lambda_m$, we have
$$\exp\left(\sum_{n\geq 1}\frac{p_n(\lambda_1,\ldots, \lambda_m)}{n}T^n\right)=\sum_{n\geq 0} h_n(\lambda_1,\ldots,\lambda_m) T^n=\prod_{i=1}^m\frac{1}{1-\lambda_iT}=\frac{1}{\det(1-\phi T)}$$
The general ungraded case can be deduced from the Zariski density of diagonalizable $\varphi$. 

To upgrade to the graded case, we simply note the well--known identities $p_n(-\lambda_1,\ldots,-\lambda_m)=(-1)^n p_n(\lambda_1,\ldots,\lambda_m)$, $h_n(-\lambda_1,\ldots,-\lambda_m)=(-1)^ne_n(\lambda_1,\ldots,\lambda_m)$ and $\sum_{n\geq 0} e_n(\lambda_1,\ldots,\lambda_m) T^n=\prod_{i=1}^m(1+\lambda_i T)$ and apply these for the odd graded pieces.

\end{proof}

\section{The orbifold measure on symmetric powers}\label{orsec}

%Let $X$ be a $k((t))$-variety and $n\geq 0$ an integer. For any algebraically closed extension $K/k$ and any partition $\lambda = (1^{a_1},2^{a_2},\dots)$ of $n$ we have a map
%\[ \pi_\lambda: \prod_i X(K((t^{1/i}))^{a_i} \to \Sym_n(X)(K((t))).   \]

%We denote the image of $\pi_\lambda$ by $\Sym_\lambda(X)(K((t)))$. Varying $K$ gives rise to a definable subassignment $\Sym_\lambda(X) \in \GDef$. 
Let $\CX / k[[t]]$ be separated, flat and of finite type with smooth $d$-dimensional generic fiber. Furthermore let $0 \in \CX_k$ be a point in the special fiber of $\CX$ and define $\Sym_n(\CX^0)$ as the definable subassignment whose $K$ points are given by $\{ x \in \Sym_n(\CX)(K[[t]]) \ |\ x_K = n[0]   \}$ as before. 

 Let $p_i:\CX^n \to \CX$ denote the projection onto the $i$-th coordinate. Given an algebraic $d$-form $\omega$ on $\CX_{k((t))}$, consider the $dn$-form $\omega_n = p_1^*\omega \wedge \dots \wedge p_n^*\omega$ on $\CX_{k((t))}^n$. For any $\sigma \in S_n$ we have $\sigma^*\omega_n = \mathrm{sgn}(\sigma)^d \omega_n$. Thus $\omega_n$ does not descend to a form on $\Sym_n\CX_{k((t))}$ unless $d$ is even.

Instead of $\omega_n$ we thus consider its tensor-square $\omega_n^{\otimes 2} \in H^0\left(\CX_{k((t))}^n,( \Omega_{\CX_{k((t))}^n/k((t))}^{dn})^{\otimes 2} \right)$. This form now descends to the \textit{orbifold form} $\omega_{orb}$ on $\Sym_n\CX_{k((t))}$ and for any choice of square root $\BL^{1/2}$ of $\BL$ we obtain, assuming integrability, a well-defined class
\[\int_{\Sym_n\CX^0} |\omega_{orb}|^{1/2} \in \calc_k[ \BL^{1/2}]. \]

 %We define the orbifold form $|\omega_{orb}|$ on $\Sym_n \CX$ by
%\begin{equation}\label{orbiform} |\omega_{orb}|  = \begin{cases} \BL^{-\frac{ord_\Delta}{2}} |\omega_{\Sym_n \CX}| \ \ &\text{ if } m=1, \\
%|\omega_{\Sym_n \CX}| \ \ &\text{ if } m \geq 2.  \end{cases}   \end{equation}

%Consider now $\CX/k[[t]]$ with smooth generic fiber and $\omega$ on $\CX_{k((t))}$ as above and let $0 \in \CX_k$. Write $\CX^0_n$ for the subassignment whose $K$-points are given by $\{ x \in \CX(K[[t^{1/n}]]) \ |\ x_K = 0\}$.

\begin{proposition}\label{orbif} Assume that $|\omega|_n$ is integrable on $\CX^0_n$ for each $n \geq 1$. Then we have
\[\int_{\Sym_n \CX^0} |\omega_{orb}|^{1/2} = \sum_{\lambda = (1^{a_1},2^{a_2},\dots) \vdash n}\BL^{-v(\lambda)} \prod_{i\geq 1} \Sym_{a_i}\left( \int_{\CX^0_i} |\omega|_i   \right),\]
where  $v(\lambda) = \frac{d}{2} \sum_{i\geq 1} a_i (i-1)$.% and  $\CX^0_n$ is the subassignment whose $K$-points are given by $\{ x \in \CX(K[[t^{1/n}]]) \ |\ x_K = 0\}$.
\end{proposition}
\begin{proof} Since everything depends only on a neighborhood of $0$ we may assume that $\CX$ is affine. Let us assume first that $0 \in \CX$ is a smooth point of $\CX/k[[t]]$. Then the completed local ring $\widehat{\CO}_{\CX,0}$ is isomorphic to $k[[t]][[x_1,\dots,x_d]]$. Since $\omega$ is generically non-vanishing we may write $\omega  = \lambda dx_1\wedge\dots \wedge dx_d$ in these local coordinates for some $\lambda \in k[[t]]$. Then $|\omega_{orb}|^{1/2} = |\lambda|^n |(dx_1\wedge\dots \wedge dx_d)_{orb}|^{1/2}$ and also $|\omega|_i = |\lambda|^i |dx_1\wedge\dots \wedge dx_d|$ on $\CX^0_i$. Thus after multiplying by $|\lambda|^{-n}$ on both sides we may assume that $\omega = dx_1\wedge\dots \wedge dx_d$. In this case the proposition reduces to the orbifold formula for $S_n$ acting on $(\BA^d)^n$ as it is proven for example in \cite[Theorem 3.6]{DenefLoeser02b}.
In the notation of \textit{loc. cit} we have
\[ \int_{\Sym_n (\BA^m)^0} |\omega_{orb}|^{1/2} = \sum_{[\sigma] \in \Conj(S_n)} \BL^{-w(\sigma)}, \]
where $\Conj(S_n)$ denotes conjugacy classes in $S_n$ and the weight $w(\sigma)$ is defined as follows: Let $r$ denote the order of $\sigma$ and $\xi$ a primitive $r$-th root of unity. Then there are unique integers $1 \leq e_{\sigma,i} \leq r$ for $1 \leq i \leq dn$ such that $\xi^{e_{\sigma,i}}$ are the eigenvalues for the action of $\sigma$ on $(\BA^d)^n$. If $\sigma$ has cycle-type given by the partition $\lambda =(1^{a_1},2^{a_2},\dots)$ then one gets the the formula
\[ w(\sigma) = \frac{d}{2}\sum_i a_i(i+1) = v(\lambda) + d\sum_i a_i.   \] 
Since in this case $\Sym_{a_i}\left( \int_{\CX^0_i} |\omega|_i   \right) = \Sym_{a_i}\BL^{-d} = \BL^{-da_i}$ the statement follows.

If $0\in \CX$ is not a smooth point, we may choose by \cite[Proposition 2.5]{NS07} a resolution of singularities, that is a regular, flat $k[[t]]$-variety $\CY$ with $\CY_k = \sum_{i} N_i E_i$ a strict normal crossing divisor, and a proper birational morphism $\pi\colon \CY \to \CX$ which is an isomorphism on generic fibers. Regularity implies that for any $K/k$, any $K[[t]]$-point of $\CY$ factors through the smooth locus of $\CY$ \cite[Proposition 3.2]{BoschLutkebohmertRaynaud90}. By \cite[Lemma 4.3, Theorem 4.5]{NS07} we may further assume that the same holds true for $K[[t^{1/i}]]$-points for all $1 \leq i \leq n$. The proposition thus reduces to the smooth case proven above. 
\end{proof}

If we apply this to $\CX^0_f$ and $\omega_f$ with $f\colon \BA^m \to \BA^1$ as in Section \ref{mizf} we get the following interpretation of the plethystic exponential of the motivic Igusa zeta function:

\begin{theorem}\label{plesym}
\[\sum_{n\geq 0} \int_{\Sym_n \CX^0_{f}} |\omega_{orb}|^{1/2} T^n =  \Exp\left( \BL^{-\frac{m-1}{2}} Z_f(\BL^{-\frac{m-3}{2}} T)\right).\]
\end{theorem}
\begin{proof} Combining Proposition \ref{orbif} and Theorem \ref{clin} we get
\begin{align*} \int_{\Sym_n \CX^0_f} |\omega_{orb}|^{1/2} &=  \sum_{\lambda = (1^{a_1},2^{a_2},\dots) \vdash n}\BL^{-v(\lambda)} \prod_{i\geq 1} \Sym_{a_i}\left( \int_{\CX^0_{f,i}} |\omega_f|_i   \right)  \\  
&=   \sum_{\lambda = (1^{a_1},2^{a_2},\dots) \vdash n}\BL^{-v(\lambda)} \prod_{i\geq 1} \Sym_{a_i}\left( \BL^{-(i+1)(m-1)} [X^0_{f,i}] \right).
\end{align*}

Thus for the generating series we get

\begin{align*}
\sum_{n\geq 0}\int_{\Sym_n \CX^0_{f}}& |\omega_{orb}|^{1/2} T^n =   \sum_{\lambda = (1^{a_1},2^{a_2},\dots) } \BL^{-\frac{m-1}{2} \sum  a_i(i-1)} \prod_{i\geq 1} \Sym_{a_i}\left( \BL^{-(i+1)(m-1)} [X^0_{f,i}] \right) T^{\sum a_i i}\\
&= \prod_{i\geq 1} \left(   \sum_{a\geq 0}  \Sym_a \left( \BL^{-\frac{(m-1)(3i+1)}{2} }[X^0_{f,i}] \right) T^{ia}    \right) \\
&= \prod_{i\geq 1} \Exp(  \BL^{-\frac{(m-1)(3i+1)}{2} }[X^0_{f,i}]T^i   )\\
&= \Exp \left( \BL^{-\frac{m-1}{2}}\sum_{i\geq 1} \BL^{-im} [X^0_{f,i}] (\BL^{-\frac{m-3}{2}} T)^i    \right) = \Exp\left( \BL^{-\frac{m-1}{2}} Z_f(\BL^{-\frac{m-3}{2}} T)\right).
 \end{align*}
\end{proof}

\begin{remark}\label{m3}If we take $f \in k[x,y,z]$, then  the generic fiber of $\CX_f$ is a smooth surface and $h:\Hilb_n(\CX_{f|k((t))}) \to \Sym_n( \CX_{f|k((t))})$ is a crepant resolution. Thus by the change of variables formula 
\[ \sum_{n\geq 0}\int_{\Hilb_n(\CX^0_{f})} |h^*\omega_{orb}|^{1/2} T^n =   \Exp\left( \BL^{-1} Z_f(T)\right).\]
It would be interesting to see, if the left hand side admits a description in term of the Hilbert scheme of the surface singularity $f=0$ similar to Theorem \ref{gelhil}.
\end{remark}

\subsection{Relation to the Gelfand form} Let $f\in k[X,Y]$ be an isolated plane curve singularity as in Section \ref{gelin}. In this section we give a precise relation relation between $|\omega_{orb}|$ and $|\omega^{gel}_n|$ on $\Sym_n \CX^0_{f}$. 

For this consider the discriminant subscheme $\tilde{\Delta} \subset (\BA^2)^n$ given by
\[\widetilde{\Delta} = \bigcup_{i\neq j} \{ x_i = x_j, y_i=y_j\}.\]

It follows from the results of Haiman, see e.g.  \cite{Haiman02a}, that the ideal $J$ defining $\widetilde{\Delta}$ equals the ideal generated by the diagonally alternating polynomials on $X_1,\ldots, X_n, Y_1,\ldots, Y_n$ and is generated by the functions $\{\Delta_M\}_{|M| = n}$ introduced in Section \ref{ha2}. Furthermore the image $\Delta \subset \Sym_n(\BA^2)$ of $\widetilde{\Delta}$ has defining ideal $J \cap k[X_1,Y_1,\dots,X_n,Y_n]^{S_n}$ which has generators $\{\Delta_M\Delta_{M'} \}_{|M| =|M'|=n}$.

We write $\widetilde{\Delta}_f$ and $\Delta_f$ for the intersection with $\CX_f^n \subset (\BA^2)^n$ and $\Sym_n(\CX_f) \subset \Sym_n(\BA^2)$ respectively.

\begin{proposition}\label{orbgel} For any $n \geq 1$ we have 
\[ \BL^{-\ord_{\Delta_f}/2} |\omega_{orb}|^{1/2} = |\omega^{gel}_n| \]
on $\Sym_n \CX^0_{f}$.
\end{proposition}

\begin{proof} The proof goes by comparing the underlying algebraic forms defining both sides when pulled back under $\pi: \CX_f^n \to \Sym_n(\CX_f)$. For the left hand side recall that $\pi^*\omega_{orb} = (p_1^*\omega_f \wedge \dots \wedge p_n^*\omega_f)^{\otimes 2}$ by definition. Furthermore for any $x \in \Sym_n\CX^0_{f}(K[[t]])$ and any lift $\tilde{x} \in \CX_f^n(\overline{K[[t]]})$ we have $\ord_{\Delta_f}(x) = \ord_\Delta(x) = 2 \ord_{\widetilde{\Delta}}(\tilde x)$ since the generators of the ideal of $\Delta$ are the pairwise products of the generators of the ideal defining $\widetilde{\Delta}$. 

For the right hand side recall that we have an open cover $\bigcup_{\lambda \vdash n} U_{f,\lambda} = \Hilb_n(\CX_f)$ and $ |\omega^{gel}_n|$ is glued together from the absolute values of the forms $\omega_\lambda^{gel} = \omega^{[n]}/ f_\lambda^{[n]*}\nu$ on $U_{f,\lambda}^{sm}$. Let $\CX'_f = \CX_f \setminus \{0\}$ denote the smooth locus of $\CX_f$. Then $\Sym_n(\CX'_f) \cong \Hilb_n(\CX'_f)$ and we can think of $U'_{f,\lambda}= U_{f,\lambda}^{sm}\cap \Hilb_n(\CX'_f)$ as an open in $\Sym_n(\CX_f)$ and hence compute $\pi^*\omega_{\lambda | U'_{f,\lambda}}^{gel}$. 

First since $\Hilb_n(\BA^2) \to \Sym_n(\BA^2)$ is crepant, it follows that $\omega^{[n]}$ pulls back to the standard form $dx_1\wedge dy_1 \wedge \dots \wedge dy_n$ on $(\BA^2)^n$. In order to compute $\pi^*f_\lambda^{[n]*}\nu$ we use that by \eqref{ftof} 

\[   ( f_\lambda^{[n]}\circ \pi)(x,y)=B_\lambda^{-1}(x,y) \begin{pmatrix}
         f(x_1,y_1)\\
         f(x_2,y_2)\\ 
         \vdots \\ 
         f(x_n,y_n) 
     \end{pmatrix}.\]
Thus $\pi^*\omega^{gel}_\lambda=\pi^*\omega^{[n]}/\pi^*f_\lambda^{[n]*}\nu$  can be written as 
\[ \pi^*\omega^{gel}_\lambda = (\Delta_\lambda(x,y) + t \alpha_1(x,y) + t^2\alpha_2(x,y) + \dots) p_1^*\omega_f\wedge \dots \wedge p_n^*\omega_f,\] 
with $\alpha_i \in K[x,y]$. Since $\pi^*\omega^{gel}_\lambda$ is $S_n$-invariant, all the $\alpha_i(x,y)$ are alternating. 

Now given any $z \in \Sym_n \CX^0_{f}(K[[t]]) \cong \Hilb_n(\CX_f^0)(K[[t]])$ there exists a $\lambda \vdash n$ such that $x \in U_{f,\lambda}(K[[t]])$. Let $(x,y) \in (\BA^2)^n(\overline{K((t))})$ be such that $\pi(x,y) = z_{|K((t))}$. Then from Lemma \ref{hev} and the fact that the ideal $J$ defining $\widetilde{\Delta}$ is generated by the functions $\{\Delta_M\}_{|M| = n}$ we finally get 
\[ |\Delta_\lambda(x,y) + t \alpha_1(x,y) + t^2\alpha_2(x,y) + \dots | = |\Delta_\lambda(x,y)| = \BL^{-\ord_{\widetilde{\Delta}}(x,y)}, \]
which finishes the proof.
\end{proof}

The motivic version of Igusa's monodromy conjecture \cite{DL98} relates the poles of $Z_f(T)$ to the roots of $\zeta_{f,0}(T)$. Hence it might be interesting to relate the integrals of $|\omega_{orb}|$ and $|\omega^{gel}_n|$ on $\Sym_n \CX^0_{f}$. Proposition \ref{orbgel} gives a way of doing so by considering the series 
\[  Q_f(s,T) = \sum_{n\geq 0} \int_{\Sym_n \CX^0_{f}} |\Delta_f|^s |\omega_{orb}|^{1/2} T^n,\]
for some formal parameter $s$. In Section \ref{cep} below we discuss a few examples showing that the computation of $Q_f(s,T)$ seems to be an interesting challenge.

\section{Relation to Floer theories}\label{rkt}
\subsection{Fixed point Floer homology}
Let $f\in \C[X,Y]$ and $f^{-1}(0)=C\subseteq \BA^2_\C$ the plane curve it defines. Denote the Milnor fiber $f^{-1}(\epsilon)\cap B_{\delta,0}$ of $f$ by $\Sigma$ and let $\varphi: \Sigma\to \Sigma$ be the symplectomorphism given by the monodromy of the Milnor fibration. The fixed point Floer homology of $\varphi^n:\Sigma\to \Sigma$ for $n>1$ is a graded $\BZ/2$-vector space $$HF_*(\varphi^n,+)$$ associated to $\varphi^n:\Sigma\to \Sigma$, only depending on the embedded contact type of the link $L$ of $f$. For the definition, see \cite[6.2.5.]{debobadillapelka}.
In this case we have

\begin{theorem}\cite[Theorem 1.1.]{arcfloer}
\label{prop:fixedpointfloer}
There is an isomorphism of graded vector spaces
$$HF_*(\varphi^n,+)\cong H_c^{*+2n+1}(X^0_{f,n})$$
 where $X^0_{f,n}$ is the $m$-th restricted contact locus as before.
\end{theorem}

The theorem is expected to hold for any hypersurface \cite{BdBLN22}. 
A fortiori, this implies an equality of Euler characteristics $\chi(HF_*(\varphi^n,+)=-\chi(X^0_{f,n})$.
%On the other hand,  it follows from \cite{DenefLoeser02b} that 
%$\chi(X^0_{f,m})=\Lambda(\varphi^m)$, the Lefschetz number of $\varphi^m$ of the monodromy on the singular cohomology of the Milnor fiber $\Sigma$.  Using Eq. \eqref{eq:monodromyzetafn} we deduce that there is an equality
%\begin{equation}\label{chihf}
%\exp\left(\sum_{n\geq 1}\frac{\chi(HF_*(\varphi^n,+)}{n}T^n \right)=\zeta_{f,0}(T),\end{equation}
%which of course also follows from the fact that the Lefschetz number of $\varphi^m$ equals the Euler characteristic of $HF_*(\varphi^m,+)$.

\subsection{Knot Floer homology}
We consider another Floer-theoretic invariant of the link $L$ defined by $f$. It is the minus version of the {\em knot Floer homology} of $L$, defined in \cite{OZ} for $\mathbb{Z}/2\BZ$-coefficients and in \cite{gridhomologybook} for $\BZ$-coefficients.

If $L$ is a link with $r$ components, the knot(/link) Floer homology is a $\BZ^{r+1}$-graded $\BZ[U]$-module
$$HFL^-(L,S^3)=\bigoplus_{v\in \BZ^r, d\in \BZ } HFL^-_d(L,S^3,v)$$
where $U$ is an operator which decreases the Maslov/homological grading by $2$ and the Alexander grading by $\underline{1}$. 
We will drop $S^3$ from the notation, as it is always the ambient $3$-manifold for us. Taking the Euler characteristic in the Alexander grading, one has 
$$\sum_{v\in\BZ^r}\chi(HFL^-(L,v))t^v=\begin{cases} \Delta(t_1,\ldots, t_r), & \text{ if } r>1 \\
\frac{\Delta(t)}{1-t}, & \text{ if } r=1\end{cases}$$

It follows from \cite[Corollary 1.6]{GorskyNemethi} and Proposition  \ref{prop:motivicequalsgel} that we have the following.
\begin{corollary}
\label{thm:motivichilb}
We have an equality
$$P_{gel}(T)=\sum_{v\in \BZ^r,d\in \BZ} \dim HFL^-_d(L,v) \BL^{d-2h(\underline{v})} T^{|v|},$$
where $h(\underline{v})=\codim_{k[[x,y]]}J_{\underline{v}}$ as in Section \ref{mps}.
%That is, the ''motivic Poincar\'e series" $P_{gel}(T)$ of the one--generator loci of the Hilbert schemes of $f$ introduced in Subsection \ref{mps} computes a specialization of the knot Floer homology of $L$.
\end{corollary}
\begin{remark}
Recall that  \cite[Corollary 1.5.3.]{GorskyNemethi}  is proved by relating both sides to the lattice homology of isolated singularities defined by Némethi. Thanks to the integral version of knot Floer homology in  \cite{gridhomologybook} the coefficient issue in \cite[Remark 2.1.2]{GorskyNemethi} can be avoided.
\end{remark} %A priori, it only works with $\ZZ/2\ZZ$-coefficients but we work with the assumption it also works with $\Q$-coefficients (see \cite[Remark 2.1.2.]{GoNe}).

\subsection{Questions and speculations}
% https://q.uiver.app/#q=WzAsMTQsWzEsMywiUF97Z2VsfShUKSJdLFsxLDUsIlxcZXhwKFpfZihUKSkiXSxbMCw0LCJRX2YocyxUKSJdLFsyLDZdLFsyLDFdLFszLDUsIlpfZihUKSJdLFszLDMsIkhfY14qKFhfe2Ysbn1eMCkiXSxbMCwyXSxbNCwyXSxbMiwwXSxbMyw2XSxbMSw2XSxbMSwxLCJIRkxeLShMLFNeMykiXSxbMywxLCJIRl8qKFxccGhpXm4sKykiXSxbMiwxLCJzPTAiLDIseyJzdHlsZSI6eyJib2R5Ijp7Im5hbWUiOiJzcXVpZ2dseSJ9fX1dLFsyLDAsInM9MS8yIiwwLHsic3R5bGUiOnsiYm9keSI6eyJuYW1lIjoic3F1aWdnbHkifX19XSxbNyw4LCJBbGdicmFpYyIsMix7ImxhYmVsX3Bvc2l0aW9uIjowLCJzdHlsZSI6eyJoZWFkIjp7Im5hbWUiOiJub25lIn19fV0sWzMsOSwiIiwwLHsic2hvcnRlbiI6eyJzb3VyY2UiOjEwfSwic3R5bGUiOnsiaGVhZCI6eyJuYW1lIjoibm9uZSJ9fX1dLFsxMCwxMSwiZXhwb25lbnRpYWwiLDAseyJvZmZzZXQiOjUsImN1cnZlIjoxLCJzaG9ydGVuIjp7InNvdXJjZSI6MzAsInRhcmdldCI6NDB9fV0sWzAsMTIsIlxccmVme30iLDIseyJsYWJlbF9wb3NpdGlvbiI6NjAsInNob3J0ZW4iOnsic291cmNlIjoyMCwidGFyZ2V0IjoyMH0sInN0eWxlIjp7InRhaWwiOnsibmFtZSI6ImFycm93aGVhZCJ9fX1dLFs2LDEzLCJcXHJlZnt9IiwyLHsibGFiZWxfcG9zaXRpb24iOjYwLCJzaG9ydGVuIjp7InNvdXJjZSI6MjAsInRhcmdldCI6MjB9LCJzdHlsZSI6eyJ0YWlsIjp7Im5hbWUiOiJhcnJvd2hlYWQifX19XSxbNSwxLCJcXHJlZnt9IiwyLHsibGFiZWxfcG9zaXRpb24iOjQwLCJzaG9ydGVuIjp7InNvdXJjZSI6MzAsInRhcmdldCI6MjB9LCJzdHlsZSI6eyJ0YWlsIjp7Im5hbWUiOiJhcnJvd2hlYWQifX19XSxbMTMsMTIsIlxccmVme30iLDIseyJsYWJlbF9wb3NpdGlvbiI6MzAsInNob3J0ZW4iOnsic291cmNlIjoyMCwidGFyZ2V0IjoyMH0sInN0eWxlIjp7InRhaWwiOnsibmFtZSI6ImFycm93aGVhZCJ9LCJib2R5Ijp7Im5hbWUiOiJkYXNoZWQifX19XV0=
\begin{figure}[H]
\[\begin{tikzcd}
	&& {} \\
	& {HFL^-(L,S^3)} & {} & {HF_*(\phi^n,+)} \\
	{} &&&& {} \\
	& {P_{gel}(T)} && {H_c^*(X_{f,n}^0)} \\
	{Q_f(s,T)} \\
	& { \Exp\left( \BL^{-\frac{1}{2}} Z_f(\BL^{\frac{1}{2}} T)\right)} && {[X_{f,n}^0]} \\
	& {} & {} & {}
	\arrow[<->,"{\ref{floereq}}"'{pos=0.3}, shorten <=9pt, shorten >=9pt, dashed, from=2-4, to=2-2]
\arrow[<->,"{\ref{Q1}}"'{pos=0.3}, shorten <=9pt, shorten >=9pt, dashed, from=4-2, to=4-4]
	\arrow["Symplectic"{pos=0},"Algbraic"'{pos=0}, no head, from=3-1, to=3-5]
	\arrow[<->,"{\ref{thm:motivichilb}}"'{pos=0.4}, shorten <=6pt, shorten >=6pt, from=4-2, to=2-2]
	\arrow[<->, "{\text{\cite{arcfloer}}}"'{pos=0.4}, shorten <=6pt, shorten >=6pt,  from=4-4, to=2-4]
	\arrow["{s=1/2}", squiggly, from=5-1, to=4-2]
	\arrow["{s=0}"', squiggly, from=5-1, to=6-2]
	\arrow[<->,"{\ref{plesym}}"'{pos=0.4}, shorten <=17pt, shorten >=11pt, from=6-4, to=6-2]
	\arrow[, shorten <=12pt, no head, from=7-3, to=1-3]
	\arrow["exponential", bend left=10, shift right=5, shorten <=40pt, shorten >=40pt, from=7-4, to=7-2]
\end{tikzcd}\]
\caption{}\label{diagram}
\end{figure}
%Corollary \ref{thm:motivichilb} relates the motivic generating series of the one-generator Hilbert schemes to the knot Floer homology. Similarly, Proposition \ref{prop:fixedpointfloer} relates the compactly supported cohomology of the restricted contact loci to fixed point Floer homology. One could similarly to Eq. \eqref{chihf} form the generating series with coefficients in graded vector spaces 
%$$\exp\left(\sum_{n\geq 1} HF_*(\varphi^n,+) \frac{T^n}{n}\right)$$
%\todo{I'm not sure what this notation means} and in the Euler characteristic limit, 
%$\zeta_{f,0}=P_{gel,\chi}(T)$ as we have seen before.

%However, before this limiting process, the two series are markedly different. One is a generating series for graded vector spaces, and the other with coefficients in a Grothendieck ring. As we know from Theorem \ref{plesym}, the motivic Igusa zeta function can be obtained by integrating the form $\omega_{orb}$ over the symmetric powers of $\CX_f^0$ and taking the plethystic logarithm. Applying $H_c^*(-)$ coefficient-wise recovers the series 
%\[ \sum_{m\geq 1} H_c^{\bullet+2m+1}(X_{f,m}^0)T^m\]
%On the other hand, $P_{gel}(T)$ is by definition obtained by integrating the form $\omega_{gel}$ over $\Sym^n\CX_f^0$. One could take $H_c^*(-)$ of the corresponding series but at least on the level of Poincar\'e series in the unibranch case, this would not yield any more information as $\Hilb^1_n(f)$ are affine spaces in this case.

In  \cite{GhigginiSpano} Ghiggini and Spano show that for fibered knots
$\widehat{HFL}(\overline{K},1-g)\cong HF_*(\varphi,\sharp)$ and more generally conjecture
that $$\widehat{HFL}(\overline{K},i)\cong PFH_{i+g}(T_\varphi,\sharp)$$ where $\widehat{HFL}(-)$ denotes the hat version of  knot Floer homology, $HF_*(\varphi,\sharp)$ is an ''intermediate" variant of fixed point Floer homology, and $PFH(T_\varphi,\sharp)$ denotes a certain version of the periodic Floer homology of (the mapping torus of) $\varphi$. We refer to {\em loc. cit.} for the precise definitions, but remark that multisets of fixed points give generators of the periodic Floer complex and more generally the theme of realizing Heegaard-Floer theory in terms of Fukaya categories of symmetric powers has been a recurring theme in the Floer theory literature, see for example \cite{Auroux}. It thus seems possible to expect a relation of the form 
\begin{equation}\label{floereq}\text{"Exponential of fixed point Floer homologies of the iterates of $\varphi$"}=HFL^-(K).\end{equation}

Combining Corollary \ref{thm:motivichilb} with \cite[Theorem 1.1.]{arcfloer} this leads to the following natural question which can be phrased in purely algebraic terms. 

\begin{question}\label{Q1} Is there an exponential-type relation between $H_c^*(X^0_{f,n})$ and the cohomology or motives of $\Hilb_n^1(f)$?
\end{question}

Let us also mention that Rossinelli  \cite{Ro24} has shown a relation between curvilinear Hilbert schemes and motivic classes of certain open subvarieties in $X^0_{f,n}$, but we don't see a direct connection with Question \ref{Q1} at the moment. 
We summarize the relations between the various invariants discussed above in Figure \ref{diagram}.

\section{Computations and examples}\label{cep}
In this section we discuss the generating series $Q_f(T)$ in the case of the smooth curve \(f=x\) and of
the simplest non-trivial singularity \(f=xy\). 

\subsection{Smooth curve}
\label{sec:smooth-curve}
The Hilbert scheme \(\Hilb_n(f)\) of the smooth line \(f(x,y)=x\)  is an affine space \(\bbA^n=\Sym_n(\bbA^1)=\Hilb^1_n(f)\) that can be identified with the space
of monic  polynomials of degree \(n\), \(y^n+\sum_ia_iy^i\). Thus the discriminant \(\Delta_f\) agrees with the classical discriminant \(\mathrm{Disc}_n\).
Moreover, since \(\Hilb_n(f)\) is smooth we have \(|\omega^{gel}|=|da_0\wedge \dots \wedge da_{n-1}|\) and the function
\(Q_f(s,T)\) is a generating function for the Igusa zeta functions of discriminants:
\[Q_f(s,T)=\sum_{n\ge 0} I_n(s)T^n,\quad I_n(s)=\int_{\bbA^n_0}|\mathrm{Disc}_n|^{s-1/2}|da_0\dots da_{n-1}|.\]
For small values of \(n\) the corresponding Igusa zeta function can be computed. For \(n=1\) the discriminant is constant and
\(I_1(s)=\LL^{-1}\).  For \(n=2\) the discriminant
is \(a_1-4a_0^2\). Thus after the change of variables \(\tilde{a}_1=a_1-4a_0^2,\tilde{a}_0=a_0\) we obtain:
\[I_2(s)=\int_{\bbA^2_0}|\tilde{a}_1|^{s-1/2}|d\tilde{a}_0d\tilde{a}_1|=(\LL-1)\frac{\LL^{s-5/2}}{\LL^{s+1/2}-1}.\]

Since the discriminant of a cubic polynomial \(x^3+px+q\) is equal to \(-4p^3-27q^2\) the computation of the integral reduces to the computation of the Igusa zeta
function for the cusp singularity. This computation can be found in many places, for example in chapter 1, \(\S 3\)  of  \cite{Chambert-LoirNicaiseSebag18}:
\[I_3(s)=\frac{\LL-1}{X^6\LL^5-1} \left(\frac{1}{\LL^2}+X^3+\frac{X^2}{\LL}+X^4\LL+\frac{\LL-1}{\LL^3(X\LL-1)} \right),\quad \LL^{s-1/2}=X.\]

The computation of the integral \(I_4(s)\) is equivalent to the computation of the Igusa zeta function for the swallow-tail singularity and we postpone the discussion 
for the forthcoming paper where we use Fulton-MacPherson space formalism to compute the integrals \(I_n\).   On other hand we check the properties of the function
\(Q_f(s,T)\) with the above computations.

Indeed, since \([\Hilb_n^1(f)_0]=1\)  we obtain \(P_{gel}(T)=\sum_{n\ge 0} (T/\LL)^n\).  As check of formula \eqref{eq:interpolate} one can verify that \(I_n(1/2)=\LL^{-n}\).
To discuss the second part of the formula \eqref{eq:interpolate} we need to compute the Igusa  zeta function for \(f\), which is by definition \eqref{eq:IgusaZ} given by
%\footnote{ Let us  point out that \(Z_f(\LL^a)\) is not equal to \(\int_{\bbA_0^2(k[[t]])}|x|^a\).}
\(Z_f(T)=\sum_{n\ge 1}(T/\LL)^n\). Thus the formula \eqref{eq:interpolate}
states that
\[\prod_{n\ge 1}(1-T^n)^{\LL^{-(n+1)/2}}=1+\sum_{n\ge 1}I_n(0)T^n.\]

The first four terms of the \(T\)-expansion are consistent with the above computations. For example the terms in front of \(T^2 \) are \(\LL^{-2}+\LL^{-3/2}\) and
the terms in front of \(T^3\) are \(\LL^{-3}+\LL^{-5/2}+\LL^{-2}\).

\subsection{Nodal singularity}
\label{sec:nodal-singularity}

The Hilbert scheme \(\Hilb_n(f)_0\) in this case is known to be
a chain of projective lines. Indeed, to describe \(\Hilb_n(f)_0\) we need to classify codimension  \(n\) ideals in the local ring of this singularity which
is equal to \(k\oplus x\cdot k[[x]]\oplus y\cdot k[[y]]\) as a \(k\)-vector space with \(xy=0\)  as the main defining relation. In details, the ideals of codimension \(n\) are
union of \(n-1\) projective lines \(\mathbb{P}^1_{k}=\{ (\alpha_kx^k+\beta_ky^{n-k},x^{k+1},y^{n-k+1})\}\),  \(\alpha_k,\beta_k\ne \vec{0}\), \(k=1,\dots,n-1\).
The points ideal for \(\beta_k=0\) is equal to the ideal for \(\alpha_{k-1}=0\) and \(\alpha_k=0\) is equal to the ideal for \(\beta_{k+1}=0\).

From the above description we see that the one generator locus \(\Hilb_n^1(f)_0\) inside \(\Hilb_n(f)_0\) is the union of \(\mathbb{G}_m\subset \mathbb{P}^1_k\) defined by the
condition \(\alpha_k\beta_k\ne 0\). Thus we have computed the \(P_{gel}(T)\):
\[P_{gel}(T)=1+(\LL-1)\sum_{n=2}^\infty(n-1)\LL^{-n} T^n=1+T^2/(\LL-T)^2. \]

The Igusa zeta function can be computed by elementary means as well, since
\([X^0_{f,n}]=(n-1)(\LL-1)\LL^{n-2}\)
\begin{equation*}
  Z_f(T)=\sum_{n\ge 1}[X^0_{f,n}]\LL^{-2n}T^n
  =(\LL-1)\LL^{-2}T^2/(1-T\LL^{-1})^2.
  \end{equation*}

  The function \(Q_f(s,T)\) interpolates between the functions \(P_{gel}(T)\) and \(Z_f(T)\) as we state in \eqref{eq:interpolate}.
  In a forthcoming manuscript, we provide and argument similar to Proposition~\ref{orbif} that 
yields a formula for the function \(Q_f(s,T)\):
  \[\int_{\Sym_n(\CX_f)}|\omega_{orb}|^{1/2}|\Delta_f|^s=\sum_{\lambda = (1^{a_1},2^{a_2},\dots) \vdash n}\BL^{-v(\lambda)}
    \int_{\prod_i \Sym_{a_i}(\CX^0_i) }|\Delta_f|^s\prod_{i\geq 1} |\omega_f|_i   .\]  

  For $n\leq 3$  the integrals can be computed by elementary means: for \(n=1\) the integral vanishes. For \(n=2\), only the term with \(\lambda=(2)\) appears with \(\nu(\lambda)=1/2\). The locus \(\CX^0_{f,2} \) consist of the pairs
  \(x,y\in t^{1/2}k[[t^{1/2}]]\) such that \(xy=t\). Hence \(|\Delta_f|\) is constant along \(\CX^0_{f,2}\) and equal \(\LL\). Thus combining this observation with the Lemma~\ref{clin} and
  \([X_{f,2}^0]=[\{x,y\in k[t]/t^3|x,y=t^2\in k[t]/t^3\}]=(\LL-1)\LL^2\) we obtain:
  \[\LL^{-\nu((2))}\int_{\CX^0_{f,2}}|\Delta_f|^s|\omega_f|_2=\LL^{-1/2}\LL^{-s}(\LL-1)\LL^{-1}.\]

  Similarly, for \(n=3\) the only non-trivial term in the formula is the term with \(\lambda=(3)\). Respectively, \(\nu((3))=1\) and the space \(\CX^0_{f,2}\) consists of two connected components
  \(\{x=t^{1/3}\varphi(t^{1/3}),y=t^{2/3}\psi(t^{1/3})\}\) and \(\{x=t^{2/3}\varphi(t^{1/3}),y=t^{1/3}\psi(t^{1/3})\}\) with \(\varphi,\psi\in k[[t^{1/3}]]\) with \(\varphi\psi=1\).
  Hence \(|\Delta_f|\) is constant along the both components and it is equal \(\LL^2\). We can again use Lemma~\ref{clin} and
  \([X_{f,3}^0]=2(\LL-1)\LL^3\) to complete the computation:
  \[\LL^{-\nu((3))}\int_{\CX^0_{f,3}}|\Delta_f|^s|\omega_f|_2=2\LL^{-1}\LL^{-2s}(\LL-1)\LL^{-1}.\]
  
  Thus we can write the start of \(T\)-expansion of the function \(Q_f(s,T)\):
  \[Q_f(s,T)=1+(\LL-1)(\LL^{-s-3/2}T^2+2\LL^{-2s-2}T^3)+\dots.\]
  If we substitute \(s=1/2\) into the expansion then we obtain the first four terms of the series \(P_{gel}(T)\). On the other hand the substitution \(s=0\) yields the first four terms of
  the expansion of
  \begin{multline*}
  \Exp(\LL^{-1/2}Z_f(\LL^{1/2}T))=\Exp((\LL-1)(\LL^{-3/2}T^2+\LL^{-2}T^3)+\dots)\\=(1-T^2)^{-\LL^{-3/2}(\LL-1)}(1-T^3)^{-2\LL^{-2}(\LL-1)}\cdot\dots \end{multline*}

% \bibliography{expanded}

\bibliographystyle{plain}

\end{document}